\PassOptionsToPackage{svgnames}{xcolor}

\documentclass[10pt,twoside]{siamart}
\usepackage[english]{babel}
\usepackage{graphicx,epstopdf,epsfig}
\usepackage{amsfonts,epsfig,fancyhdr,graphics, hyperref,amsmath,amssymb}
\usepackage{multirow}
\newcommand{\Names}{Fabio Durastante, Isabella Furci}
\newcommand{\Title}{Spectral Analysis of Saddle--point Matrices from Optimization problems with Elliptic PDE Constraints}

\newtheorem{Theorem}{Theorem}
\newtheorem{Lemma}{Lemma}
\newtheorem{Proposition}{Proposition}

\theoremstyle{definition}
\newtheorem{Definition}{Definition}
\newtheorem{Remark}{Remark}

\renewcommand{\theequation}{\thesection.\arabic{equation}}
\renewtheorem{theorem}{Theorem}[section]
\usepackage{tikz,pgfplots}
\usetikzlibrary{decorations.markings}
\usetikzlibrary{shapes.geometric}
\pgfdeclarelayer{edgelayer}
\pgfdeclarelayer{nodelayer}
\pgfsetlayers{edgelayer,nodelayer,main}
\tikzstyle{none}=[inner sep=0pt]
\tikzstyle{wh}=[circle,fill=White,draw=Black,line width=0.8 pt]
\tikzstyle{rn}=[circle,fill=Red,draw=Black,line width=0.8 pt]
\tikzstyle{gn}=[circle,fill=Lime,draw=Black,line width=0.8 pt]
\tikzstyle{yn}=[circle,fill=Yellow,draw=Black,line width=0.8 pt]
\tikzstyle{simple}=[-,draw=Black,line width=2.000]
\tikzstyle{arrow}=[-,draw=Black,postaction={decorate},decoration={markings,mark=at position .5 with {\arrow{>}}},line width=2.000]
\tikzstyle{tick}=[-,draw=Black,postaction={decorate},decoration={markings,mark=at position .5 with {\draw (0,-0.1) -- (0,0.1);}},line width=2.000]

\usepackage{subfig}
\usepackage{booktabs}
\hypersetup{breaklinks=true}

\newcommand{\bx}[1]{{\bf #1}}
\DeclareMathOperator*{\diag}{diag}
\DeclareMathOperator*{\esssup}{ess\,sup}
\DeclareMathOperator*{\essinf}{ess\,inf}

\begin{document}

\bibliographystyle{plain}

\setcounter{page}{1}

\thispagestyle{empty}

 \title{Spectral Analysis of Saddle--point Matrices from Optimization problems with Elliptic PDE Constraints\thanks{This work was partially supported by INdAM-GNCS
 project ``Tecniche innovative per problemi di algebra
 lineare'' (2018), and by the Tor Vergata University ``MISSION: SUSTAINABILITY'' project
 “NUMnoSIDS”, CUP E86C18000530005.}}

\author{
Fabio Durastante\thanks{Istituto per le Applicazioni del Calcolo ``Mauro Picone''. Consiglio Nazionale delle Ricerche, Napoli, Italy.
(f.durastante@na.iac.cnr.it).}
\and
Isabella Furci\thanks{Department of Mathematics and Informatics. University of Wuppertal, Wuppertal, Germany. (furci@uni-wuppertal.de).}}

\markboth{\Names}{\Title}

\maketitle

\begin{abstract}
The main focus of this paper is the characterization and exploitation of the asymptotic spectrum of the saddle--point matrix sequences arising from the discretization of optimization problems constrained by elliptic partial differential equations. We uncover the existence of an hidden structure in these matrix sequences, namely, we show that these are indeed an example of Generalized Locally Toeplitz (GLT) sequences. We show that this enables a sharper characterization of the spectral properties of such sequences than the one that is available by using only the fact that we deal with saddle--point matrices. Finally we exploit it to propose an optimal preconditioner strategy for the GMRES, and Flexible--GMRES methods.
\end{abstract}


\section{Introduction}

Linear systems with saddle--point matrices arises in a wide context of applications and have attracted a great deal of attention~\cite{benzi2005numerical,MR3564863}. In general form, they can be simply stated as the family of linear systems where the left--hand side is given by a block--matrices of the form
\begin{equation}\label{eq:the_saddle}
\mathcal{A}_N = \begin{bmatrix}
A & B_1^T \\
B_2 & -C
\end{bmatrix}, \qquad A \in \mathbb{R}^{q \times q}, B_1,B_2 \in \mathbb{R}^{p \times q}, \, C \in \mathbb{R}^{p \times p}. 
\end{equation}
We are interested here in the analysis of their spectral properties in the very specific context of the discretized version of optimal constraint problems~\cite{troltzsch2010optimal}
\begin{equation}\label{eq:thegeneralproblem}
\left\lbrace\begin{array}{rl}
\displaystyle \min_{y,u} J(y,u) = & \displaystyle \frac{1}{2}\|y - y_d\|_{L^2(\Omega)}^2  + \frac{\alpha}{2} \|u\|_{L^2(\Omega)}^2,\\
\text{ such that } & \begin{array}{ll}
e(y,u) = 0, & \text{ in } \Omega,\\
y = f, & \text{ on } \partial\Omega_D,\\
\frac{\partial y}{\partial \mathbf{n}} = g, & \text{ on } \partial\Omega_N,\\
\end{array}
\end{array}\right.
\end{equation}
where, $\alpha > 0$ is a fixed constant that acts as a Tikhonov regularization parameter, $J$ is a cost functional, $\Omega \subset \mathbb{R}^d$ is the domain of both the state $y$ and the control $u$, and $\partial\Omega_D$ and $\partial\Omega_N$ are two disjoint sets that represent the Dirichlet and Neumann boundary respectively and have the whole boundary as union. 

Spectral properties of the general case~\eqref{eq:the_saddle} have been indeed thoroughly  analyzed~\cite{MR1168084,murphy2000note,benzi2006eigenvalues,MR2220678,MR1802361,gould2009spectral,sesana2013spectral,bergamaschi2012eigenvalue} under several hypotheses on the blocks of $\mathcal{A}_N$, e.g., $B_1 = B_2 = B$, $C$ semipositive definite, $A$ symmetric and positive definite, and so on. {The goal of the latter works has been to provide a} sharp localization bounds for their spectrum, and exploit them to devise efficient iterative solvers for such problems. Here we focus on a less general objective, i.e., we intend to exploit finer information on the structure of the blocks of~\eqref{eq:the_saddle}, a knowledge coming from the coupling of the source problem~\eqref{eq:thegeneralproblem} and its discretization, to give an asymptotic description of the spectrum of the matrices~$\{\mathcal{A}_N\}_N$. Specifically, we show that the saddle--point form of $\mathcal{A}_N$ obtained from~\eqref{eq:the_saddle} hides inside another structure, namely, that the sequence of matrices $\{\mathcal{A}_N\}_N$ is a Generalized Locally Toeplitz (GLT) sequence~\cite{glt1,SerraLibro1}. This enables us to obtain a sharper localization of its asymptotic spectrum. Furthermore, we use this characterization to suggest an effective preconditioning strategy for such problems. We stress that an approach of this type has already been exploited for both the saddle--point matrices obtained from a two--dimensional linear elasticity--type problem in~\cite{MR3689933}, and partially explored in~\cite{MR3711990,cipolladurastante} for a constrained optimization problem where the constraints $e(y,u)$ were Fractional Differential Equations.

The paper is therefore divided as follows, in Section~\ref{sec:discretization} we describe the discrete form of~\eqref{eq:thegeneralproblem} fully specifying the sequence of matrices $\{\mathcal{A}_N\}_N$. In Section~\ref{sec:spectralanalysis} we recall the essential tools needed for working with GLT sequences and apply them to our problem, while in Section~\ref{sec:efficientsolution} we exploit them to devise an efficient preconditioning strategy. {In Section~\ref{sec:numerical_experiments} we substantiate our claims with some numerical examples, and give conclusions in Section~\ref{sec:conclusion_and_future_developments}.}

\section{From the Continuous Problem to the Saddle--point sequence $\{\mathcal{A}_N\}_N$}
\label{sec:discretization}

The first point we need to answer is how we obtain the sequence of saddle--point matrices from~\eqref{eq:thegeneralproblem}, indeed a way of doing so is going through its Langrangian formulation. Thus, we find the Lagrangian of~\eqref{eq:thegeneralproblem} as
\begin{equation}\label{eq:thelagrangian}
\mathcal{L}(y,u,p) = J(y,u) - \langle p, e(y,u) \rangle_{W^*,W},
\end{equation}
where $e(y,u)$ represents the PDE constraint as an operator between the Banach spaces $Y\times U$ and $W$, and $p$ is the Adjoint status between the space $W$ and its dual $W^{*}$ acting as Lagrange multiplier. Indeed, a solution for the original constrained optimization problem~\eqref{eq:thegeneralproblem} is a stationary point for the Lagrangian~\eqref{eq:thelagrangian}. To obtain such stationary point $(\hat{y},\hat{u},\hat{p}) \in Y\times U\times W^{*}$ we require that the G\^ateaux derivative with respect to each of the variables of~\eqref{eq:thelagrangian} is zero, i.e.,
\begin{equation*}
\begin{aligned}
\mathcal{L}'_y(\hat{y},\hat{u},\hat{p})\mathfrak{h} =&\, J'_y(\hat{y},\hat{u})\mathfrak{h}  - \langle \hat{p}, e'_y(\hat{y},\hat{u})\mathfrak{h}  \rangle_{W^*,W} = 0, &\qquad \forall \,\mathfrak{h}  \in Y,\\
\mathcal{L}'_u(\hat{y},\hat{u},\hat{p})\mathfrak{w}  =&\, J'_u(\hat{y},\hat{u})\mathfrak{w}  - \langle \hat{p}, e'_u(\hat{y},\hat{u})\mathfrak{w}  \rangle_{W^*,W} = 0, &\qquad \forall \,\mathfrak{w}  \in U,\\
\mathcal{L}'_p(\hat{y},\hat{u},\hat{p}) =&\, e(\hat{y},\hat{u}) = 0.
\end{aligned}
\end{equation*}
These are called, in general, the first order optimality conditions or the Karush-Kuhn-Tucker conditions (KKT-conditions) for Problem~\eqref{eq:thegeneralproblem}. Finally, for obtaining such characterization we have to fully specify the operator $e(y,u)$, and consequently all the functional spaces $Y,U$, and $W$. The  {prototypical elliptic} problem in this class is represented by the Poisson distributed control
\begin{equation}\label{eq:theproblem}
\left\lbrace\begin{array}{rl}
\displaystyle \min_{y,u} J(y,u) = & \displaystyle \frac{1}{2}\|y - y_d\|_{L^2(\Omega)}^2  + \frac{\alpha}{2} \|u\|_{L^2(\Omega)}^2,\\
\text{ such that } & \begin{array}{ll}
- \nabla^2 y = u + z, & \text{ in } \Omega,\\
y = f, & \text{ on } \partial\Omega_D,\\
\frac{\partial y}{\partial \mathbf{n}} = g, & \text{ on } \partial\Omega_N,\\
\end{array}
\end{array}\right.
\end{equation}
{where $z$ represents the forcing term.}

The KKT conditions for problem~\eqref{eq:theproblem} are expressed as
\begin{equation}\label{eq:kkt-conditions}
\begin{array}{ll}
\displaystyle\left\lbrace\begin{array}{ll}
- \nabla^2 y = u + z, & \text{ in } \Omega,\\
y = f, & \text{ on } \partial\Omega_D,\\
\frac{\partial y}{\partial \mathbf{n}} = g, & \text{ on } \partial\Omega_N.\\
\end{array}\right. & \qquad \text{(State equation)}\\\\
\displaystyle\left\lbrace\begin{array}{ll}
- \nabla^2 p = y-y_d, & \text{ in } \Omega,\\
y = 0, & \text{ on } \partial\Omega_D,\\
\frac{\partial y}{\partial \mathbf{n}} = 0, & \text{ on } \partial\Omega_N.\\
\end{array}\right. & \qquad \text{(Adjoint equation)}\\\\
\displaystyle\alpha u + p = 0. & \qquad \text{(Gradient condition)}
\end{array}
\end{equation}
By posing $\hat{p} = -p$ and choosing $v \in H^{1}_0(\Omega)$ we can rewrite conditions~\eqref{eq:kkt-conditions} in weak form~as:
\begin{align}
\int_{\Omega} \nabla u \cdot \nabla v \,dx = & \int_{\Omega} u v \,dx,+ \int_{\Omega} z v \,dx,\nonumber\\
\int_{\Omega} \nabla \hat{p} \cdot \nabla v \,dx = & \int_{\Omega} (y_d - y)v \,dx,\label{eq:integral_optimality_system}\\
\alpha \int_{\Omega} u v \,dx - \int_{\Omega} \hat{p} v \,dx = & 0.\nonumber		
\end{align}
Finally, the sequence $\{\mathcal{A}_N\}$ is obtained by fixing a Finite Element (FEM) approximation of the optimality system~\eqref{eq:integral_optimality_system}. This means fixing a space $V_{0,\mathbf{n}}(\Omega_\mathbf{n})$  with $V_{0,\mathbf{n}} = \operatorname{Span}\{\phi_1,\ldots,\phi_{N\mathbf{(n)}}\} \subset H^{1}_0(\Omega)$ over a mesh $\Omega_\mathbf{n}$ on the domain $\Omega$ thus obtaining the linear system
\begin{equation}\label{eq:the_linear_system}
\bar{\mathcal{A}}_N\mathbf{x} \equiv \left[\begin{array}{cc|c}
\bar{M} & O & \bar{K}^T \\
&    & \\
O & \alpha \bar{M} & - \bar{M} \\
&    & \\
[-0.1em]      \hline        &    & \\

\bar{K} & -\bar{M} & O
\end{array}\right]\begin{bmatrix}
\mathbf{y} \\ 
\\
\mathbf{u} \\ 
\\
\\
\mathbf{p}
\end{bmatrix} = \begin{bmatrix}
M \mathbf{y}_d \\ 
\\
\mathbf{0} \\
\\
\mathbf{z}
\end{bmatrix} \equiv \bar{\mathbf{b}},
\end{equation}
where
\begin{equation}\label{eq:thematrixblocks}
(\bar{M})_{i,j} = \int_{\tau_h} \phi_i \phi_j d\mathbf{x}, \qquad (\bar{K})_{i,j} = \int_{\tau_h} \nabla \phi_i \cdot \nabla \phi_j d\mathbf{x}, 
\end{equation}
are the usual (scaled)  {mass} and  {stiffness} matrices, and $O$ is the zero matrix of order $N(\textbf{n})= n_1n_2 \ldots n_d$.

\subsection{Triangular Lagrangian Elements}

To completely specify the linear system~\eqref{eq:the_linear_system} we need to precise both the mesh $\Omega_{N(\mathbf{n})}$ and the basis functions $\{\phi_j\}_{j=1}^{N(\mathbf{n})}$, i.e., chose the  {element} defining our discretization. We focus here on {nodal Lagrangian elements~\cite[Chapter~5]{MR2322235} of degree $p$}. These are built starting from $\mathbb{P}_p$, the vector space of polynomials $q(x_1,x_2)$ with scalar coefficients of $\mathbb{R}^2$ in $\mathbb{R}$ of degree less than or equal to $p$, 
\begin{equation*}
\mathbb{P}_p = \left\lbrace q(x_1,x_2) = \sum_{0 \leq i + j \leq p}c_{i,j} x_1^{i}x_2^{j}, \quad c_{i,j} \in \mathbb{R} \right\rbrace.
\end{equation*}
That is indeed a vector space of dimension $\dim \mathbb{P}_p = \frac{1}{2}(p+1)(p+2)$. Then an homogeneous triangulation $\Omega_{N(\mathbf{n})}$ of the unit square domain $\Omega = [0,1]^2$ is considered, i.e., a mesh consisting in 2D triangular cells $\tau_h$ with straight sides, and a lattice $\Sigma_p$ of nodes $\{\mathbf{N}_i\}_{i=1}^{\dim \mathbb{P}_p}$ on each triangle; see Figure~\ref{fig:triangular_elements}. 
\begin{figure}[htbp]
	\centering
	\subfloat[$p=1$]{\begin{tikzpicture}[scale = 0.6]
		\begin{pgfonlayer}{nodelayer}
		\node [style=wh] (0) at (-13, 5) {\small 3};
		\node [style=wh] (1) at (-13, -0) {\small 1};
		\node [style=wh] (2) at (-8, -0) {\small 2};
		\end{pgfonlayer}
		\begin{pgfonlayer}{edgelayer}
		\draw [style=simple] (0) to (1);
		\draw [style=simple] (1) to (2);
		\draw [style=simple] (2) to (0);
		\end{pgfonlayer}
		\end{tikzpicture}}\hfill
	\subfloat[$p=2$]{
		\begin{tikzpicture}[scale = 0.6]
		\begin{pgfonlayer}{nodelayer}
		\node [style=wh] (0) at (-13, 5) {\small 3};
		\node [style=wh] (1) at (-13, -0) {\small 1};
		\node [style=wh] (2) at (-8, -0) {\small 2};
		\node [style=wh] (3) at (-10.75, -0) {\small 4};
		\node [style=wh] (4) at (-13, 2.5) {\small 6};
		\node [style=wh] (5) at (-10.5, 2.5) {\small 5};
		\end{pgfonlayer}
		\begin{pgfonlayer}{edgelayer}
		\draw [style=simple] (0) to (1);
		\draw [style=simple] (1) to (2);
		\draw [style=simple] (2) to (0);
		\end{pgfonlayer}
		\end{tikzpicture}}\hfill
	\subfloat[$p=3$]{
		\begin{tikzpicture}[scale = 0.6]
		\begin{pgfonlayer}{nodelayer}
		\node [style=wh] (0) at (-13, 5) {\small 3};
		\node [style=wh] (1) at (-13, -0) {\small 1};
		\node [style=wh] (2) at (-8, -0) {\small 2};
		\node [style=wh] (3) at (-11.25, -0) {\small 4};
		\node [style=wh] (4) at (-13, 1.5) {\small 9};
		\node [style=wh] (5) at (-9.5, 1.5) {\small 6};
		\node [style=wh] (6) at (-9.5, -0) {\small 5};
		\node [style=wh] (7) at (-13, 3.25) {\small 8};
		\node [style=wh] (8) at (-11.25, 3.25) {\small 7};
		\node [style=wh] (9) at (-11.25, 1.5) {\small 10};
		\end{pgfonlayer}
		\begin{pgfonlayer}{edgelayer}
		\draw [style=simple] (0) to (1);
		\draw [style=simple] (1) to (2);
		\draw [style=simple] (2) to (0);
		\end{pgfonlayer}
		\end{tikzpicture}
	}
	\caption{Nodes $\mathbf{N}_i$ for the linear ($p=1$), quadratic $(p=2)$ and cubic ($p=3$) Lagrange polynomials on a  triangle}\label{fig:triangular_elements}
\end{figure}
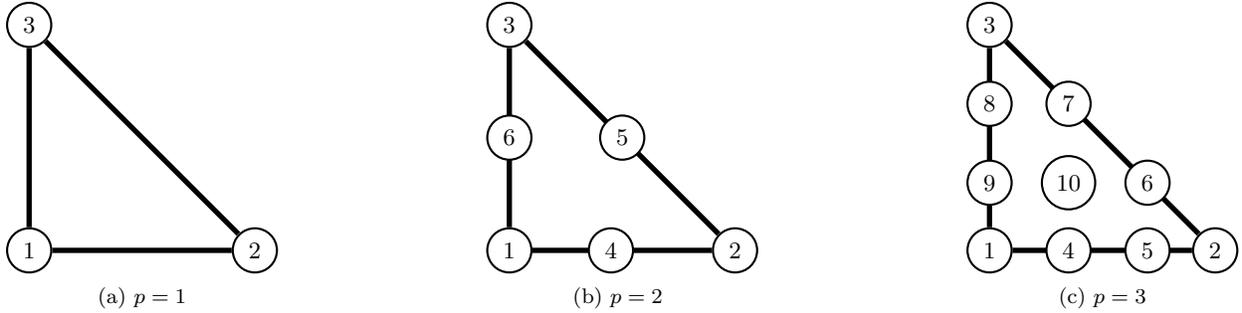

By this construction, every polynomial $q \in \mathbb{P}_p$ is uniquely determined by its values at the points $\{\mathbf{N}_i\}_{i=1}^{\dim \mathbb{P}_p}$. The finite element method
for triangular Lagrange $\mathbb{P}_p$ elements is then built on the discrete finite dimensional space
\begin{equation*}
V_\mathbf{n}^{p} = \{v \in \mathcal{C}^{0}(\Omega) \, v|_{\tau_h} \in \mathbb{P}_p, \quad \tau_h \in \Omega_{N(\mathbf{n})} \} \subset H^{1},
\end{equation*}
and its subspace
\begin{equation*}
V_{0,\mathbf{n}}^{p} = \{v \in V_\mathbf{n}^{p}, \quad v = 0\text{ on }\partial \Omega \} \subset H^{1}_0.
\end{equation*}	
We call  {degrees of freedom} of a function $v \in  V_\mathbf{n}^{p}$ the set of the values of $v$ at the nodes~$\mathbf{N}_j$ on the entire mesh, then the space $V_{0,\mathbf{n}}^{p}$ has exactly the dimension corresponding to the number of  {internal degrees of freedom}, i.e., excluding the nodes on $\partial \Omega$. For our model grid we find that the degrees of freedom are $N( \mathbf{n}) =n_1n_2= (p n_x+1)(p n_y + 1)$, where $n_x$ and $n_y$ are the number of elements in the $x$ and $y$ direction, respectively.  {{Thus} the dimension $N$ of the matrix in~\eqref{eq:thematrixblocks} will be equal to $3N(\mathbf{n})$}. The matrices~\eqref{eq:thematrixblocks} are then constructed by means of the opportune Gauss quadrature formulas, and in terms of the Lagrange basis functions $\{\phi_i\}_{i=1}^{N(\mathbf{n})}$. For all the discussion, and computation in the paper we deal with the matrices generated for such elements by the FEniCS library (v.2018.1.0)~\cite{AlnaesBlechta2015a,LoggOlgaardEtAl2012a}.

\section{Spectral analysis of the resulting sequence of saddle point matrices}
\label{sec:spectralanalysis}

This section is devoted to the attainment of a characterization of the spectra of a suitable scaling $\{{\mathcal{A}}_N\}_N$ of the sequence of matrices $\{\bar{\mathcal{A}}_N\}_N$ in~\eqref{eq:the_linear_system}. Specifically, we are going to answer to the following questions,
\begin{enumerate}
	\item[\emph{Q1}] can we individuate some (possibly sharp) intervals containing the spectrum with respect to $N$?
	\item[\emph{Q2}] For a given $N$ how many eigenvalues are in each interval?
	\item[\emph{Q3}] What is the relation between the condition number of a suitably preconditioned matrix sequence and the value of the regularization parameter $\alpha$?
\end{enumerate}
As we mentioned in the introduction, there exist classical localization results for the eigenvalues of a symmetric saddle--point matrix, like the $\mathcal{A}_N$ in~\eqref{eq:the_saddle}.
\begin{Theorem}[Rusten and Winther~\cite{MR1168084}]\label{thm:rusten_winther} 
	Given~$\mathcal{A}_N$ in~\eqref{eq:the_saddle}, assume $A$  is symmetric and positive definite, $B_1 = B_2 = B$ has full rank, and $C=0$. Let $\mu_1$ and $\mu_n$ denote the largest and smallest eigenvalues of $A$, and let $\sigma_1$ and $\sigma_m$ denote the largest and smallest singular values of $B$. Then the spectrum of $\mathcal{A}_N$ is contained in 
	\begin{equation*}
	I^{-} \cup I^{+},
	\end{equation*}
	where 
	\begin{align*}
	I^{-} =  \left[ \frac{1}{2}\left(\mu_n - \sqrt{\mu_n^2 + 4\sigma_1^2}\right); \frac{1}{2}\left(\mu_1 - \sqrt{\mu_1^2 + 4\sigma_m^2}\right) \right],
	I^{+} =  \left[\mu_n ; \frac{1}{2} \left(\mu_1 + \sqrt{\mu_1^2 + 4\sigma_1^2}\right)\right].
	\end{align*}
\end{Theorem}
This bound is indeed very general and versatile, since it requires only information on the symmetry/defini\-teness of the diagonal blocks, and on the rank of the extradiagonal ones.  {It can be used to obtain an estimate of the condition number of $\mathcal{A}_N$ as function of $N$ in a straightforward way. To this end, an even sharper result can be obtained by means of \cite[Theorem~1(c)]{MR2220678} that permits to characterize exactly the eigenvalues with the largest and the smallest module.} Nevertheless, by exploiting further information on the blocks, we show that finer answers to our question are indeed possible. Specifically, we are going to individuate three disjoint intervals $I_0^{-}$, $I_1^{+}$, and $I_2^{+}$ containing the spectrum of the scaled version of $\bar{\mathcal{A}}_N$ {, we show that this choice is not arbitrary, and that it stems directly from the structure of the problem, and the selection of the discretization scheme.}

In Section~\ref{sec:background_and_definitions}, we start recalling the tools we use, and then we deploy them to achieve these results in Section~\ref{Study_BN_perm}.

\subsection{Background and definitions}
\label{sec:background_and_definitions}

Throughout this paper, we use the following notation. Let $\mathbb{C}^{s\times s}$ be the linear space of the complex $s\times s$ matrices and let $\textbf{f}:G\to\mathbb{C}^{s\times s}$, with $G\subseteq\mathbb R^\ell$, $\ell\ge 1$, measurable set. We say that $\textbf{f}$ belongs to $L^1(G)$ (resp. is measurable) if all its components $\textit{f}_{ij}:G\to\mathbb C,\ i,j=1,\ldots,s,$ belong to $L^1(G)$ (resp. are measurable).
We denote by {$\mathcal{I}_d$ the $d$-dimensional
cube $(-\pi,\pi)^d$} and define $L^1(d,s)$ as the linear space of $d$-variate functions $\textbf{f}:\mathcal{I}_d\to \mathbb{C}^{s\times s}$, $\textbf{f}\in
L^1(\mathcal{I}_d)$.

Moreover we indicate by $\{\mathcal{A}_{N}\}_{{\bf n}\in\mathbb{N}^d}$, or simply $\{\mathcal{A}_{N}\}_{{\bf n}}$, the matrix sequence whose elements are the matrices $\mathcal{A}_{N}$ of dimensions $N \times N = N(s,\textbf{n})\times N(s,\textbf{n})$, with $N(s,\textbf{n})=sN(\textbf{n})=s n_1n_2 \ldots n_d$,  $\textbf{n}=(n_1,n_2,\ldots,n_d)$.

\begin{Definition}\label{def-multilevel}
	Let the Fourier coefficients of a given function $\textbf{f}\in L^1(d,s)$ be defined as
	\begin{align}\label{fhat}
	\hat {\textbf{f}_{\bf j}}:=\frac1{(2\pi)^d}\int_{ {\mathcal{I}_d}}\textbf{f}(\boldsymbol{\theta})\, {\rm e}^{\, -{ {\iota}}\,\left\langle \,{\bf j}\,,\,\boldsymbol{\theta}\right\rangle}\ {\rm d}\boldsymbol{\theta}\in\mathbb{C}^{s\times s},
	\qquad {\bf j}=(j_1,\ldots,j_d)\in\mathbb Z^d,\ \ \  {{\iota}^2=-1},
	\end{align}
	where {$\left\langle { \bf j},\boldsymbol{\theta}\right\rangle=\sum_{t=1}^dj_t\theta_t$ } and the integrals in \eqref{fhat} are computed componentwise.
	
	Then, the ${\bf n}$th
	Toeplitz matrix associated with $\textbf{f}$ is the matrix of order $N(s,\textbf{n})$ given by
	\begin{equation}\label{tep_multi_block}
	T_{\bf n}(\textbf{f}) =\sum_{\bf j=-(\bf n-\bf e)}^{\bf n-\bf e} J_{n_1}^{j_1} \otimes \cdots\otimes J_{n_d}^{j_d}\otimes \hat{\textbf{f}_{\bf j}}.
	\end{equation}
	where $\bx{e}=(1,\ldots,1)\in\mathbb{N}^d, \,\bx{j}=(j_1,\ldots,j_d)\in\mathbb{N}^d$ and $ J^{j_{\xi}}_{n_\xi}$  is the $n_{\xi} \times n_{\xi}$ matrix whose $(i,l)$th entry equals 1 if $(i-l)=j_{\xi}$ and $0$ otherwise.
	
	The set $\{T_{\bf n}(\textbf{f})\}_{{\bf n}}$ (with ${\bf n}\in\mathbb N^d$) is
	called the family of $d$-level Toeplitz matrices generated by $\textbf{f}$, that in turn is referred to as the generating function or the symbol of 
	$\{T_{\bf n}(\textbf{f})\}_{{\bf n}}$.
	
	{Moreover from (\ref{fhat}) the symbol can be expressed via the Fourier series  
		\begin{equation}\label{eq:symbol}
		\textbf{f}(\boldsymbol{\theta})= \sum_{\textbf{j}=-\boldsymbol{\infty}}^{\boldsymbol{\infty}}\hat{\textbf {f}}_{\textbf{j}}e^{{\iota}\left\langle {\bf j},\boldsymbol{\theta}\right\rangle}.
		\end{equation}
	}
\end{Definition}

In order to deal with low--rank/small--norm perturbations and to show that they do not
affect the symbol of a Toeplitz sequence, we introduce the definition of spectral
distribution in the sense of the eigenvalues and of the singular values for a
generic matrix-sequence $\{\mathcal{A}_{N}\}_{{\bf n}\in\mathbb{N}^v}$, $v\ge 1$, and then the notion of GLT algebra.

\begin{Definition}\label{def-distribution}
	Let $\textbf{f}:G\to\mathbb{C}^{s\times s}$ be a measurable function, defined on a measurable set $G\subset\mathbb R^\ell$ with $\ell\ge 1$,
	$0<\mu_\ell(G)<\infty$. Let $\mathcal C_0(\mathbb K)$ be the set of continuous functions with compact support over $\mathbb
	K\in \{\mathbb C, \mathbb R_0^+\}$ and let $\{\mathcal{A}_{N}\}_{{\bf n}\in\mathbb{N}^v}$, $v\ge 1$, be a sequence of matrices with
	eigenvalues $\lambda_j(\mathcal{A}_{N})$, $j=1,\ldots,N$ and singular
	values $\sigma_j(\mathcal{A}_{N})$, $j=1,\ldots,N$.
	
	\begin{itemize}
		\item $\{\mathcal{A}_{N}\}_{{\bf n}\in\mathbb{N}^v}$ is {\em distributed as the pair
			$(\textbf{f},G)$ in the sense of the eigenvalues,} in symbols
		$$\{\mathcal{A}_{N}\}_{{\bf n}\in\mathbb N^v}\sim_\lambda(\textbf{f},G),$$ if the following limit relation holds for all $F\in\mathcal C_0(\mathbb C)$:
		\begin{align}\label{distribution:sv-eig}
		\lim_{{\bf n}\to\infty}\frac{1}{N}\sum_{j=1}^{N}F(\lambda_j(\mathcal{A}_{N}))=
		\frac1{\mu_\ell(G)}\int_G \frac{\displaystyle \sum_{i=1}^s F\bigg(\left(\lambda^{(i)}(\textbf{f})\right)(\boldsymbol{\theta})\bigg)}{s} {\rm d}\boldsymbol{\theta}.
		\end{align}
		\item $\{\mathcal{A}_{N}\}_{{\bf n}\in\mathbb N^v}$ is {\em distributed as the pair
			$(\textbf{f},G)$ in the sense of the singular values,} in symbols
		$$\{\mathcal{A}_{N}\}_{{\bf n}\in\mathbb N^v}\sim_\sigma(\textbf{f},G),$$ if the following
		limit relation holds for all $F\in\mathcal C_0(\mathbb R_0^+)$:
		\begin{align}\label{distribution:sv-eig-bis}
		\lim_{{\bf n}\to\infty}\frac{1}{N}\sum_{j=1}^{N}F(\sigma_j(\mathcal{A}_{N}))=
		\frac1{\mu_\ell(G)}\int_G \frac{\displaystyle \sum_{i=1}^s F\bigg(\left(\sigma^{(i)}(\textbf{f})\right)(\boldsymbol{\theta})\bigg)}{s} {\rm d}\boldsymbol{\theta}.
		\end{align}
	\end{itemize}
	In this setting the expression ${{\bf n}\to\infty}$ means that every component of the vector ${\bf n}$ tends to infinity, that is,
	$\displaystyle \min_{i=1,\ldots,v} n_i\to\infty$.
\end{Definition}
\begin{Remark}\label{rem_distr}
	We denote by $\lambda^{(1)}(\textbf{f}), \ldots,\lambda^{(s)}(\textbf{f})$ and by $\sigma^{(1)}(\textbf{f}),
	\ldots,\sigma^{(s)}(\textbf{f})$ the eigenvalues and the singular values of a
	$s\times s$ matrix-valued function $\textbf{f}$, respectively. If $\textbf{f}$ is
	smooth enough, an informal interpretation of the limit relation
	\eqref{distribution:sv-eig} (resp. \eqref{distribution:sv-eig-bis})
	is that when the matrix-size of $\mathcal{A}_{N}$ is sufficiently large,
	then $N/s$ eigenvalues (resp. singular values) of $\mathcal{A}_{N}$ can
	be approximated by a sampling of $\lambda^{(1)}(\textbf{f})$ (resp. $\sigma^{(1)}(\textbf{f})$)
	on a uniform equispaced grid of the domain~$G$. Analogously each following
	$N/s$ eigenvalues (resp. singular values) can be approximated by an equispaced sampling
	of the relative $\lambda^{(j)}(\textbf{f})$ (resp. $\sigma^{(j)}(\textbf{f})$), $j=2,\ldots,s$, in the domain.
\end{Remark}
\begin{Remark} \label{sec:computations_with_the_eigenvalue_functions}
	To perform the sampling in Remark~\ref{rem_distr} computing a closed analytical expression of any of the eigenvalue functions of $\textbf{f}$ is not the most effective procedure. It is costly and, essentially, useless since for $q=1,\ldots s$ we can provide an ``exact'' evaluation of $\lambda^{(q)}(\mathbf{f})$ at the grid points $\{\boldsymbol{\theta}_{\bf n}=(\theta_1^{(j)},\theta_2^{(k)})\}_{j,k=0}^{n-1}$ without actually computing the analytical expression. Indeed the ``exact'' evaluation for $d=2$ case is achieved by 
	\begin{enumerate}
		\item sampling $\mathbf{f}$ at $\boldsymbol{\theta}_{ \bf n-e}=(\theta_{n-1}^{(j)},\theta_{n-1}^{(k)})$, $j,k=0,\ldots, n-1$, and thus obtain $n^2$ $s\times s$ matrices, $A_{j,k},$ $j,k=0,\ldots,n-1$;
		\item for each $j,k=0,\ldots, n-1$, compute the $s$ eigenvalues of $A_{j,k}$, $\lambda_q(A_{j,k})$, $q=1,\ldots,s$; 
		\item for a fixed $q=1,\ldots,s $, the evaluation of $\lambda^{(q)}(\mathbf{f})$ at $\boldsymbol{\theta}_{\bf n-e}$, $j,k=0,\ldots, n-1,$ is given by $\lambda_q(A_{j,k})$, $j,k=0,\ldots,n-1$.
	\end{enumerate}
\end{Remark}

\subsubsection{Spectral analysis of Hermitian (block)  Toeplitz sequences: distribution results}

We collect here some classical results concerning the distribution of Hermitian (block) Toeplitz sequences from \cite{GSz,Tillinota}, that we will use extensively in the following.
\begin{Theorem}[Grenander and Szeg\H{o}~\cite{GSz}]\label{szego}
	Let $f\in L^1(d,1)$ be a real-valued function with $d\ge 1$. Then, $$\left\{T_{\bf n}(f)\right\}_{{\bf n}\in\mathbb N^d}\sim_\lambda(f,\mathcal{I}_d).$$
\end{Theorem}
In the case where $\textbf{f}$ is a Hermitian matrix-valued function, according to Tilli~\cite{Tillinota}, the previous theorem can be extended as follows:
\begin{Theorem}[Tilli~\cite{Tillinota}]\label{szego-herm} 
	Let $\textbf{f}\in L^1(d,s)$ be a Hermitian matrix-valued function with $d\ge 1,s\ge 2$. Then, $$\{T_{\bf n}(\textbf{f})\}_{{\bf n}\in\mathbb N^d}\sim_\lambda~(\textbf{f},\mathcal{I}_d).$$
\end{Theorem}
\begin{Remark}\label{rem_simmetria}
	If $\{T_{\bf n}(\textbf{f})\}_{{\bf n}\in\mathbb N^d}$ is such that each $T_{\bf n}(\textbf{f})$ is symmetric with real symmetric blocks, then the symbol has the additional property that
	\[\textbf{f}(\pm\theta_1,\ldots,\pm\theta_d)\equiv \textbf{f}(\theta_1,\ldots,\theta_d),\quad\forall(\theta_1,\ldots,\theta_d)\in \mathcal{I}_d^+=
	[0,\pi]^d,\] and therefore Theorem \ref{szego-herm} can be restated as
	\[\{T_{\bf n}(\textbf{f})\}_{{\bf n}\in\mathbb N^d}\sim_\lambda~(\textbf{f},\mathcal{I}_d^+).\]
	
\end{Remark}

\subsubsection{GLT sequences: operative features}

We list here some properties and operative features from the theory of GLT sequences in their block form; refer to~\cite{glt2,axioms7030049,SerraLibro2} for a full account of the GLT theory. 
\begin{description}
	\item[GLT1] Each GLT sequence has a singular value symbol $\textbf{f}(\textbf{x},\boldsymbol{\theta})$ for $(\textbf{x},\boldsymbol{\theta})\in [0,1]^d\times [-\pi,\pi]^d$ according to the second Item in Definition \ref{def-distribution} with $\ell=2d$. 
	If the sequence is Hermitian, then
	the distribution also holds in the eigenvalue sense.
	If $\{\mathcal{A}_N\}_N$ has a GLT symbol $\textbf{f}(\textbf{x},\boldsymbol{\theta})$ we will write $\{\mathcal{A}_N\}_N\sim_{\textsc{glt}} \textbf{f}(\textbf{x},\boldsymbol{\theta})$.
	\item[GLT2] The set of GLT sequences form a $*$-algebra, i.e., it is closed under linear combinations, products,
	inversion (whenever the symbol is singular, at most, in a set of zero Lebesgue measure), and conjugation. Hence, the sequence
	obtained via algebraic operations on a finite set of given GLT sequences is still a GLT sequence and its symbol is
	obtained by performing the same algebraic manipulations on the corresponding symbols of the input GLT sequences.
	
	\item[GLT3] Every Toeplitz sequence generated by an $L^1(d,s)$ function $\textbf{f}=\textbf{f}(\boldsymbol{\theta})$ is a GLT sequence and its symbol
	is $\textbf{f}$, with the specifications reported in item {\bf GLT1}. We note that the function $\textbf{f}$ does not depend on the
	space variables $\textbf{x}\in [0,1]^d$.
	
	\item[GLT4] Every sequence which is distributed as the constant zero in the singular value sense is a GLT sequence with
	symbol~$0$. In particular:
	\begin{itemize}
		\item every sequence in which the rank divided by the size tends to zero, as the matrix size tends to infinity;
		\item every sequence in which the trace-norm (i.e., sum of the singular values) divided by the size tends to zero, as the matrix size tends to infinity.
	\end{itemize}	 
	\item[GLT5] If $\{\mathcal{A}_N\}_N \sim_{\rm GLT}\kappa$ and the matrices $\mathcal{A}_N$ are such that $\mathcal{A}_N=\mathcal{X}_N+\mathcal{Y}_n$, where
	\begin{itemize}
		\item every $\mathcal{X}_N$ is Hermitian,
		\item the spectral norms of $\mathcal{X}_N$ and $\mathcal{Y}_N$ are uniformly bounded with respect to~$N$,
		\item the trace-norm of $\mathcal{Y}_N$ divided by the matrix size $N$ converges to~0,
	\end{itemize}
	then the distribution holds in the eigenvalue sense.
\end{description}
We highlight that from the previous properties follows that a sequence of Toeplitz matrices is, up to low-rank corrections, a GLT sequence whose symbol is not affected by the low-rank perturbation.
\begin{Theorem}\cite[Section 8.4]{SerraLibro1}\label{th:esercizio_libro}
	Let $\{A_N\}_N$ be a sequence of Hermitian matrices such that $\{A_N\}_N\sim_{ GLT}\kappa$, and let $\{P_N\}_N$ be a sequence of Hermitian positive definite matrices such that $\{P_N\}_N\sim_{GLT}\xi$ and $\xi\ne0$ a.e. Then 
	\[ \{P_N^{-1}A_N\}_N\sim_{\rm GLT}\xi^{-1}\kappa,\qquad \{P_N^{-1}A_N\}_N\sim_{\sigma,\,\lambda}(\xi^{-1}\kappa,\mathcal{I}^d). \]
\end{Theorem}	

\subsection{Spectral Analysis of the Sequence $\{\mathcal{A}_N\}_N$}
\label{Study_BN_perm}

We can now use the introduced tools to perform the spectral analysis of the matrix sequence $\{\bar{\mathcal{A}}_N\}_N$ {, assuming that $n=n_1=n_2$, $p=1$}. For studying it is easier to consider the equivalent distribution given by the following symmetric diagonal scaling
\begin{equation}
\mathcal{A}_N = \mathcal{D}^{(1)}_N \bar{\mathcal{A}}_N \mathcal{D}^{(2)}_N = \begin{bmatrix}
h^4  {M} & O &  {K}^T \\
O & \alpha  {M}  & - {M} \\
{K} & - {M} & O \\
\end{bmatrix}, \qquad h=\frac{1}{n+1},\label{eq:spectrally_analyzed_matrix}
\end{equation}
with
\begin{equation*}
\mathcal{D}^{(1)}_N  = \begin{bmatrix}
h^2 I_{n^2}  & O &  O \\
O & I_{n^2} & O \\
O & O & I_{n^2}\\
\end{bmatrix}, \qquad \mathcal{D}^{(2)}_N  = \begin{bmatrix}
I_{n^2}  & O &  O \\
O & \frac{1}{h^2}I_{n^2} & O \\
O & O & \frac{1}{h^2}I_{n^2}\\
\end{bmatrix}.
\end{equation*}
From the discretization of the Section \ref{sec:discretization}, the elements of the matrix $\bar{M}$ depend on $n$ as $1/(n+1)^2$. Hence, the effect of the proposed scaling permits to eliminate the dependence of $h^2$ of the elements in $\bar{M}$, which, for $n$ large, would make the matrix $\mathcal{A}_N$ ill-conditioned.

	In particular the matrices {${M} = \frac{1}{h^2}\bar{M}= T_\mathbf{n}(m)$, ${K} =\bar{K}= T_\mathbf{n}(\kappa)$} are $n^2\times n^2$ bi-level Toeplitz matrices with generating functions 
	\begin{equation}\label{eq:m-bi-toeplitz}
	m(\theta_1,\theta_2) = \frac{\cos \left(\theta_1\right)}{6}+\frac{\cos \left(\theta_2\right)}{6}+\frac{1}{6} \cos \left(\theta_1+\theta_2\right)+\frac{1}{2}, 
	\end{equation}
	and 
	\begin{equation}\label{eq:k-bi-toeplitz}
	\kappa(\theta_1,\theta_2) = -2 \cos \left({\theta} _1\right)-2 \cos \left(\theta_2\right)+4.
	\end{equation}
	{We stress that in this case the matrices $M$ and $K$ are real and symmetric. A property that we will exploit the theoretical analysis, nevertheless we keep the notation $K^T$ for the (1,3) block of the matrix $\mathcal{A}_N$ for two reasons. On one side, for being consistent with the continuous setting, in which the adjoint is usually explicitly expressed. On the other, to keep the analogy with Section~\ref{sec:Study_BN_perm_non_herm} in which we will discuss the usage of the advection-diffusion equation as constraint.}

\begin{Theorem}\label{thm:spectral_magic}
	The matrix sequence $\{\mathcal{A}_N\}_N$ in~\eqref{eq:spectrally_analyzed_matrix} is  distributed in the sense of the Eigenvalues as
	\begin{equation}\label{eqn.symbol.f_reduced}
	\textbf{f}(\theta_1 , \theta_2 ) =   \hat {\textbf{f}}_{(0,0)}+ 2 \hat {\textbf{f}}_{(0,-1)}\left(\cos\theta_1+\cos\theta_2 \right)+2  \hat {\textbf{f}}_{(-1,-1)}\left(\cos(\theta_1+\theta_2)\right),
	\end{equation}
	i.e., $\{\mathcal{A}_N\}_N\sim_\lambda(\textbf{f},[0,\pi]^2)$, where
	\begin{align}\nonumber
	\hat {\textbf{f}}_{(0,0)} =\begin{bmatrix}
	0 & 0 & 4\\
	0 & \frac{\alpha}{2}& -\frac{1}{2}\\
	4&-\frac{1}{2} & 0
	\end{bmatrix}, \quad 
	\hat {\textbf{f}}_{(1,1)} =\hat {\textbf{f}}_{(-1,-1)}=\begin{bmatrix}
	0 & 0 & 0\\
	0 & \frac{\alpha}{12}& -\frac{1}{12}\\
	0&-\frac{1}{12} & 0
	\end{bmatrix},\\ \label{eq:symbol_coefficients}
	\hat {\textbf{f}}_{(-1,0)} =\hat {\textbf{f}}_{(0,-1)}=\hat{ \textbf{f}}_{(0,1)} =\hat {\textbf{f}}_{(1,0)}=\begin{bmatrix}
	0 & 0 & -1\\
	0 & \frac{\alpha}{12}& -\frac{1}{12}\\
	-1&-\frac{1}{12} & 0 
	\end{bmatrix}.
	\end{align}
\end{Theorem}

\begin{proof}
	Let $\mathbf{e}_i$, $i=1,\ldots,N$ be the $i$th column of the identity matrix of size $N$, we can define a proper $N\times N$ permutation matrix,  $\Pi=[P_1|P_2|P_3]$, $P_l\in \mathbb{R}^{N\times n^2},\, l=1,2,3$,  such that the $k$th column of $P_l\, l=1,2,3$, is $e_{l+3(k-1)}$. The matrix $\Pi$ transforms $\mathcal{A}_N$ as 
	\begin{equation}\label{eq:bn_permuted_an}
	B_N=\Pi \mathcal{A}_N \Pi^T=T_{\textbf{n}}(\textbf{f})+E_{\textbf{n}},
	\end{equation}
	where
	\begin{itemize}
		\item $T_{\textbf{n}}(\textbf{f})$ is the bi-level $3\times 3$ block Toeplitz
		$
		T_{\bf{n}}(\textbf{f}) =\left[\hat{\textbf{f}}_{\textbf{i}-\textbf{j}}\right]_{\textbf{i},\textbf{j}=\textbf{e}}^{\bf n}\in\mathbb{C}^{N\times N}
		$
		generated by $\textbf{f}:[-\pi,\pi]^2\rightarrow \mathbb{C}^{3\times 3}$ as in~\eqref{eq:symbol},
		\item $E_\textbf{n}$ is a small-norm matrix, with $||E_\textbf{n} ||<C, $  $C$ constant depending on the bandwidths of $B_N$ and $N^{-1}\|E_\textbf{n} \|_1\to0$.
	\end{itemize}
	This is a congruence transformation, thus if we find the distribution of the sequence $\{B_N\}_N$ we found also the distribution for the sequence $\{\mathcal{A}_N\}_N$. Let us observe that the nonzero entries of $T_{\bf n}(\textbf{f})=[\hat{ \textbf{f}}_{\textbf{i}-\textbf{j}}]_{\textbf{i},\textbf{j}=\textbf{e}}^{\textbf{n}}$ correspond to the indexes $\textbf{i}=(i_1,i_2),\textbf{j}=(j_1,j_2)$ satisfying
	\[\{|i_1-j_1|+|i_2-j_2|\le1\} \cup \{i_1=i_2=j_1=j_2=1\}\cup\{i_1=i_2=j_1=j_2-1\},\] 
	as shown in equation~\eqref{toeplitz}, for $\textbf{n}=(3,3)$ we find $T_{\bf n}(\textbf{f})$
	\begin{equation}\label{toeplitz}
	{\arraycolsep=1pt\def\arraystretch{0.5}
		\left[\begin{array}{ccc|ccc|ccc}
		
		\hat{\textbf{f}}_{(0,0)} & \hat{\textbf{f}}_{(0,-1)} & 0 &\hat{\textbf{f}}_{(-1,0)} & \hat{\textbf{f}}_{(-1,-1)} & 0 & 0 & 0 & 0  \\
		&         &    &  &  &    &  & &\\ 
		\hat{\textbf{f}}_{(0,1)} & \hat{\textbf{f}}_{(0,0)} & \hat{\textbf{f}}_{(0,-1)}& 0 &\hat{\textbf{f}}_{(-1,0)} & \hat{\textbf{f}}_{(-1,-1)} & 0 & 0 & 0\\
		&         &    &  &  &    &  & &\\
		
		0 & \hat{\textbf{f}}_{(0,1)} & \hat{\textbf{f}}_{(0,0)} & 0 & 0 & \hat{\textbf{f}}_{(-1,0)}&0& 0 & 0\\
		&         &    &  &  &    &  & &\\
		\hline
		&         &    &  &  &    &  & &\\
		\hat{\textbf{f}}_{(1,0)} & 0 & 0 &\hat{\textbf{f}}_{(0,0)} & \hat{\textbf{f}}_{(0,-1)} & 0 & \hat{\textbf{f}}_{(-1,0)} & \hat{\textbf{f}}_{(-1,-1)}& 0  \\
		&         &    &  &  &    &  & &\\
		\hat{\textbf{f}}_{(1,1)} & \hat{\textbf{f}}_{(1,0)} & 0 &  \hat{\textbf{f}}_{(0,1)} & \hat{\textbf{f}}_{(0,0)} & \hat{\textbf{f}}_{(0,-1)} & 0 &\hat{\textbf{f}}_{(-1,0)}& \hat{\textbf{f}}_{(-1,-1)}\\
		&         &    &  &  &    &  & &\\ 
		0 & \hat{\textbf{f}}_{(1,1)} & \hat{\textbf{f}}_{(1,0)} &  0 &\hat{\textbf{f}}_{(0,1)} & \hat{\textbf{f}}_{(0,0)} & 0 & 0 & \hat{\textbf{f}}_{(-1,0)} \\
		&         &    &  &  &    &  & &\\   
		\hline
		&         &    &  &  &    &  & &\\
		0 & 0 & 0 & \hat{\textbf{f}}_{(1,0)} & 0 & 0 & \hat{\textbf{f}}_{(0,0)} &\hat{\textbf{f}}_{(0,-1)} & 0  \\
		&         &    &  &  &    &  & &\\
		0 & 0 & 0 &  \hat{\textbf{f}}_{(1,1)} & \hat{\textbf{f}}_{(1,0)} & 0 & \hat{\textbf{f}}_{(0,1)} &\hat{\textbf{f}}_{(0,0)} & \hat{\textbf{f}}_{(0,-1)}\\
		&         &    &  &  &    &  & &\\
		0 & 0 & 0 & 0 &\hat{\textbf{f}}_{(1,1)} & \hat{\textbf{f}}_{(1,0)}  &  0 &\hat{\textbf{f}}_{(0,1)} & \hat{\textbf{f}}_{(0,0)}\\
		\end{array}\right]}
	\end{equation}
	Therefore,  { from~\eqref{eq:symbol}}, the generating function $\textbf{f}$ is given by the finite sum
	\begin{equation}
	\begin{split}
	\textbf{f}(\theta_1 , \theta_2 ) &=  \hat {\textbf{f}}_{(0,0)}+  \hat {\textbf{f}}_{(-1,0)} e^{-\mathbf i \theta_1}+  \hat{ \textbf{f}}_{(0,-1)} e^{-\mathbf i \theta_2}+  \hat {\textbf{f}}_{(1,0)} e^{\mathbf i \theta_1}+  \hat{ \textbf{f}}_{(0,1)} e^{\mathbf i \theta_2}+\\&+ \hat{ \textbf{f}}_{(-1,-1)} e^{-\mathbf i( \theta_1+\theta_2)}+\hat{ \textbf{f}}_{(1,1)} e^{\mathbf i( \theta_1+\theta_2)},
	\end{split}
	\label{eqn.symbol.f}
	\end{equation}
	where $\hat {\textbf{f}}_{(0,0)}, \hat {\textbf{f}}_{(-1,0)},\hat {\textbf{f}}_{(0,-1)},\hat{ \textbf{f}}_{(1,0)},\hat{ \textbf{f}}_{(0,1)}, \hat{ \textbf{f}}_{(1,1)}, \hat{ \textbf{f}}_{(-1,-1)}\in\mathbb{R}^{3\times 3}$, that is $\textbf{f}$ is a linear trigonometric polynomial in the variables $\theta_1$ and $\theta_2$ with matrix coefficients from~\eqref{eq:symbol_coefficients}.
	{Moreover, using the equalities in (\ref{eq:symbol_coefficients}), the symbol in} 	\eqref{eqn.symbol.f} can be readily simplified as
	{
		\begin{equation*}
		\begin{split}
		\textbf{f}(\theta_1 , \theta_2 ) &=  \hat {\textbf{f}}_{(0,0)}+  \hat {\textbf{f}}_{(0,-1)} e^{-\mathbf i \theta_1}+  \hat{ \textbf{f}}_{(0,-1)} e^{-\mathbf i \theta_2}+  \hat {\textbf{f}}_{(0,-1)} e^{\mathbf i \theta_1}+  \hat{ \textbf{f}}_{(0,-1)} e^{\mathbf i \theta_2}+\\&+ \hat{ \textbf{f}}_{(-1,-1)} e^{-\mathbf i( \theta_1+\theta_2)}+\hat{ \textbf{f}}_{(-1,-1)} e^{\mathbf i( \theta_1+\theta_2)}\\
		&= \hat {\textbf{f}}_{(0,0)}+ \hat {\textbf{f}}_{(0,-1)} ( e^{-\mathbf i \theta_1} +   e^{\mathbf i \theta_1}     + e^{-\mathbf i \theta_2}+e^{\mathbf i \theta_2})+ \hat{ \textbf{f}}_{(-1,-1)}(e^{-\mathbf i( \theta_1+\theta_2)}+ e^{\mathbf i( \theta_1+\theta_2)})\\
		=&  \hat{ \textbf{f}}_{(0,0)}+ 2 \hat {\textbf{f}}_{(0,-1)}\left(\cos\theta_1+\cos\theta_2 \right)+2  \hat {\textbf{f}}_{(-1,-1)}\left(\cos(\theta_1+\theta_2)\right).
		\end{split}
		\end{equation*}
	}
	
	{Note, from the latter, that}
	\begin{equation*}
	\textbf{f}^T(\theta_1 , \theta_2 )=  \textbf{f}(\theta_1 , \theta_2 ),
	\end{equation*}
	thus $\textbf{f}$ is a symmetric matrix-valued function which implies that $T_{\bf n}(\textbf{f})$ is a symmetric matrix. By Theorem~\ref{szego-herm}, we conclude that
	\begin{equation}\label{distr_Tn}
	\{T_{\bf n}(\textbf{f})\}_{{\bf n}}\sim_\lambda(\textbf{f},[-\pi,\pi]^2).
	\end{equation}
	While, from \textbf{GLT3}, we know that $\{T_{\bf n}(\textbf{f})\}_{{\bf n}}$ is a GLT sequence with symbol~$\textbf{f}$.
	Moreover, let us observe that $\{E_{\bf n}\}$ is a zero--distributed sequence hence $\{E_{\bf n}\}_{{\bf n}}\sim_\sigma (\textbf{0},\mathcal{I}_2^+)$.
	Indeed,  { $E_{\bf n}$ is the permutation of a matrix that in block position (1,1) collects all the terms that contains the scaling $h^4$, deriving from the (1,1) block of $\mathcal{A}_N$, and $0$ anywhere else. Then it can be written as} $E_{\bf n}=h^4 \tilde{E}_{\bf n}.$
	
	Since the trace norm $\|\cdot\|_1$ of $ \tilde{E}_{\bf n}$ is equal to a constant $C$ independent on ${\bf n}$,  { we have}
	\begin{equation*}
	\lim_{{\bf n}\rightarrow\infty} N^{-1}||E_{\bf n}||_1=\lim_{{\bf n}\rightarrow\infty} N^{-1}\sum_{i=1}^N \sigma_i(E_{\bf n})\le \lim_{{\bf n}\rightarrow\infty} N^{-1} \sigma_{\rm max} (E_{\bf n})N=0,
	\end{equation*}
	and hence the zero--distribution follows from~\textbf{GLT4}. 
	In addition, from \textbf{GLT1} and the fact that $E_{\bf n}$ is Hermitian, $\{E_{\bf n}\}_{{\bf n}}\sim_\lambda (\textbf{0},\mathcal{I}_2^+)$.
	
	The conclusion of the Theorem is then achieved by applying \textbf{GLT2} and~\eqref{distr_Tn}, since this proves that $\{T_{\bf n}(\textbf{f})+E_{\bf n}\}_{{\bf n}\in\mathbb N^2}$ is a GLT sequence with symbol $\mathbf{f}$, i.e., $\{\mathcal{A}_N\}_N\sim_{\rm GLT}\textbf{f}$.
	Consequently, by recalling that $T_{\bf n}(\textbf{f})+E_{\bf n}$ is real symmetric for every $\bf n$ and using  \textbf{GLT1}, we
	deduce that the distribution result holds in the sense of the eigenvalues
	\begin{equation}\label{eig_distr}
	\{B_N\}_N\sim_\lambda(\textbf{f},[-\pi,\pi]^2).
	\end{equation}
	Furthermore, since each $B_N$ is symmetric and its blocks are symmetric and real, then $\textbf{f}$ is such that
	$\textbf{f}(\pm\theta_1,\pm\theta_2)\equiv \textbf{f}(\theta_1,\theta_2)$, $\forall(\theta_1,\theta_2)\in [0,\pi]^2$ and therefore (\ref{eig_distr}) can be rephrased~as
	\begin{equation}\label{eig_distr_piu}
	\{B_N\}_N\sim_\lambda(\textbf{f},\mathcal{I}_2^+).
	\end{equation}
\end{proof}

We can now find a first answer to the questions \emph{Q1} and \emph{Q2}. For $N$ sufficiently large,~let
\begin{equation*}
\lambda_1 (B_N) \le \lambda_2 (B_N)\le \ldots \le \lambda_N (B_N).
\end{equation*}
be the eigenvalues of $B_N$ from~\eqref{eq:bn_permuted_an}, i.e., of $\mathcal{A}_N$. By Remark~\ref{rem_distr}, with $s=3$, and equation~\eqref{eig_distr_piu}, we discover that $N/3=n^2$ eigenvalues of $B_N$, up to a number of outliers infinitesimal in the dimension, can be approximated by a sampling of $\lambda^{(1)}(\textbf{f})$ on an opportune grid (see the following discussion). The next $N/3$ on the second one and the last $n^2$ on the sampling of $\lambda^{(3)}(\textbf{f})$. Moreover, obtaining the following proposition, as a specialized version of Theorem~\ref{thm:rusten_winther}, is straightforward.
\begin{Proposition}\label{pro:eigenbound}
	{
		Let $m_i=\essinf_{\mathcal{I}^+_2} \lambda^{(i)}  (\textbf{f}(\boldsymbol{\theta} ))$ and $M_i=\esssup_{\mathcal{I}^+_2} \lambda^{(i)}  (\textbf{f}(\boldsymbol{\theta} ))$ be the essential infimum and essential supremum of $ \lambda^{(i)}  (\textbf{f}(\boldsymbol{\theta} ))$ respectively, for $i=1,2,3$. Then, for $N$ sufficiently large, } the spectrum $\lambda(\mathcal{A}_N)$ of the matrix sequence $\{\mathcal{A}_N\}_N$ is contained in three intervals
	\begin{equation*}
	\begin{split}
	\lambda(\mathcal{A}_N) \subset I_0^{-} \cup I_1^{+} \cup I_2^{+} = & (\essinf_{\mathcal{I}^+_2} \lambda^{(1)} (\textbf{f}(\boldsymbol{\theta} )),  \esssup_{\mathcal{I}^+_2} \lambda^{(1)}  (\textbf{f}(\boldsymbol{\theta} ))] \\& \cup(\essinf_{\mathcal{I}^+_2} \lambda^{(2)}  (\textbf{f}(\boldsymbol{\theta} )),  \esssup_{\mathcal{I}^+_2} \lambda^{(2)}  (\textbf{f}(\boldsymbol{\theta} ))]\\& \cup[\essinf_{\mathcal{I}^+_2} \lambda^{(3)}  (\textbf{f}(\boldsymbol{\theta} )),  \esssup_{\mathcal{I}^+_2} \lambda^{(3)}  (\textbf{f}(\boldsymbol{\theta} ))) \\
	= & (m_1,M_1] \cup (m_2,M_2] \cup [m_3,M_3),
	\end{split}
	\end{equation*}
	for $\mathcal{I}^+_2 = [0,\pi]^2$.
\end{Proposition}

\begin{proof}
	{From the definition of $\textbf{f}$ in~\eqref{eqn.symbol.f_reduced}}, $\forall\,  (\theta_1,\theta_2) \in [0,\pi]^2$,  { and matching with the classical analysis for saddle--point matrices in Theorem~\ref{thm:rusten_winther}, we find}
	\begin{equation}\label{lambda_1vsall}
	\left( \lambda^{(1)} (\textbf{f})\right)(\theta_1,\theta_2)<0 \le\left(\lambda^{(2)}   (\textbf{f})\right)(\theta_1,\theta_2)<\left(\lambda^{(3)}   (\textbf{f})\right)(\theta_1,\theta_2),
	\end{equation}
	i.e.,
	\begin{equation}\label{M_3}
	M_1<m_2, \qquad
	M_2<m_3.
	\end{equation}
	and
	\begin{equation}\label{eqn:sets}
	\begin{split}
	\esssup_{\mathcal{I}^+_2}\lambda^{(1)}  (\textbf{f}(\boldsymbol{\theta} ))&\le \essinf_{\mathcal{I}^+_2} \lambda^{(2)}  (\textbf{f}(\boldsymbol{\theta})),\\
	\esssup_{\mathcal{I}^+_2}\lambda^{(2)}  (\textbf{f}(\boldsymbol{\theta}))&\le \essinf_{\mathcal{I}^+_2} \lambda^{(3)}  (\textbf{f}(\boldsymbol{\theta})).
	\end{split}
	\end{equation}
	{From ~\cite[Theorem~2.3]{locspsb1}, we know that the thesis holds true for $T_{\bf n}(\textbf{f})$ and, from the relation $\{\mathcal{A}_N\}_N\sim_\lambda(\textbf{f},[0,\pi]^2)$ of Theorem \ref{thm:spectral_magic}, we have that asymptotically the inclusion in (\ref{eqn:sets}) is valid, also involving the small norm correction. }
\end{proof}

To deliver an actual numerical estimate for these bounds what we need is a reasonable approximation of the eigenvalue functions $\lambda^{(l)} ({\bf f})$, $l=1,2, 3$, following the procedure from Remark ~\ref{sec:computations_with_the_eigenvalue_functions} and exploiting Theorem~\ref{thm:spectral_magic}, we define the following equispaced grid on~$\mathcal{I}^+_2$
\begin{equation*}
\boldsymbol{\theta}_{\bf n-e}=\left\{(\theta_{n-1}^{(j)},\theta_{n-1}^{(k)})=\left(\frac{j\pi}{n},\frac{k\pi}{n}\right), \qquad j,k=0, \ldots, n-1\right\},
\end{equation*}
and consider the following $n^2$ Hermitian matrices of size $3\times3$
\begin{equation}
A_{j,k}:=\textbf{f}(\theta_{n-1}^{(j)},\theta_{n-1}^{(k)}), \quad j,k=0, \ldots, n-1.
\end{equation}
Ordering in ascending way the eigenvalues of $A_{j,k}$
\begin{equation*}
\lambda_1 (A_{j,k}) \le \lambda_2  (A_{j,k}) \le \lambda_3  (A_{j,k}), \quad j,k=0,\ldots,n-1,
\end{equation*}
for any $l=1,2,3$, an evaluation of $\lambda^{(l)}(\mathbf{f})$ at $(\theta_1^{(j)},\theta_2^{(k)})$ is given by $\lambda_l  (A_{j,k})$, $j,k=1,\ldots,n$.  For a fixed $l$, we denote the vector of all eigenvalues $\lambda_{l}  (A_{j,k})$, $j,k=0,\ldots,n-1$ as $\mathbf{P}^{(n)}_l$ , i.e.,
\begin{equation*}
\mathbf{P}^{(n)}_l:=\left[\lambda_l (A_{0,0}), \lambda_l  (A_{0,1}), \ldots ,\lambda_l  (A_{n-1,n-1})\right],
\end{equation*}
and by $\mathbf{P}^{(n)}$ the vector of all eigenvalues $\lambda_l  (A_{j,k})$, $j,k=0,\ldots,n-1$ varying $l$, i.e.,
\begin{equation*}
\mathbf{P}^{(n)}:=\left[\lambda_1 (A_{0,0}),\ldots, \lambda_1  (A_{n-1,n-1}),\ldots,\lambda_3 (A_{0,0}),\ldots,\lambda_3  (A_{n-1,n-1})\right].
\end{equation*}

Note that, refining the grid by increasing $n$, we can provide the evaluation of the eigenvalue functions of $\textbf{f}$ in a larger number of
grid points: numerical evidences of this fact are reported in Figure~\ref{n=5_n=10andn=6_n=12}, 
\begin{figure}[htbp]
	\subfloat[$n=5$]{\includegraphics[width=.5\columnwidth]{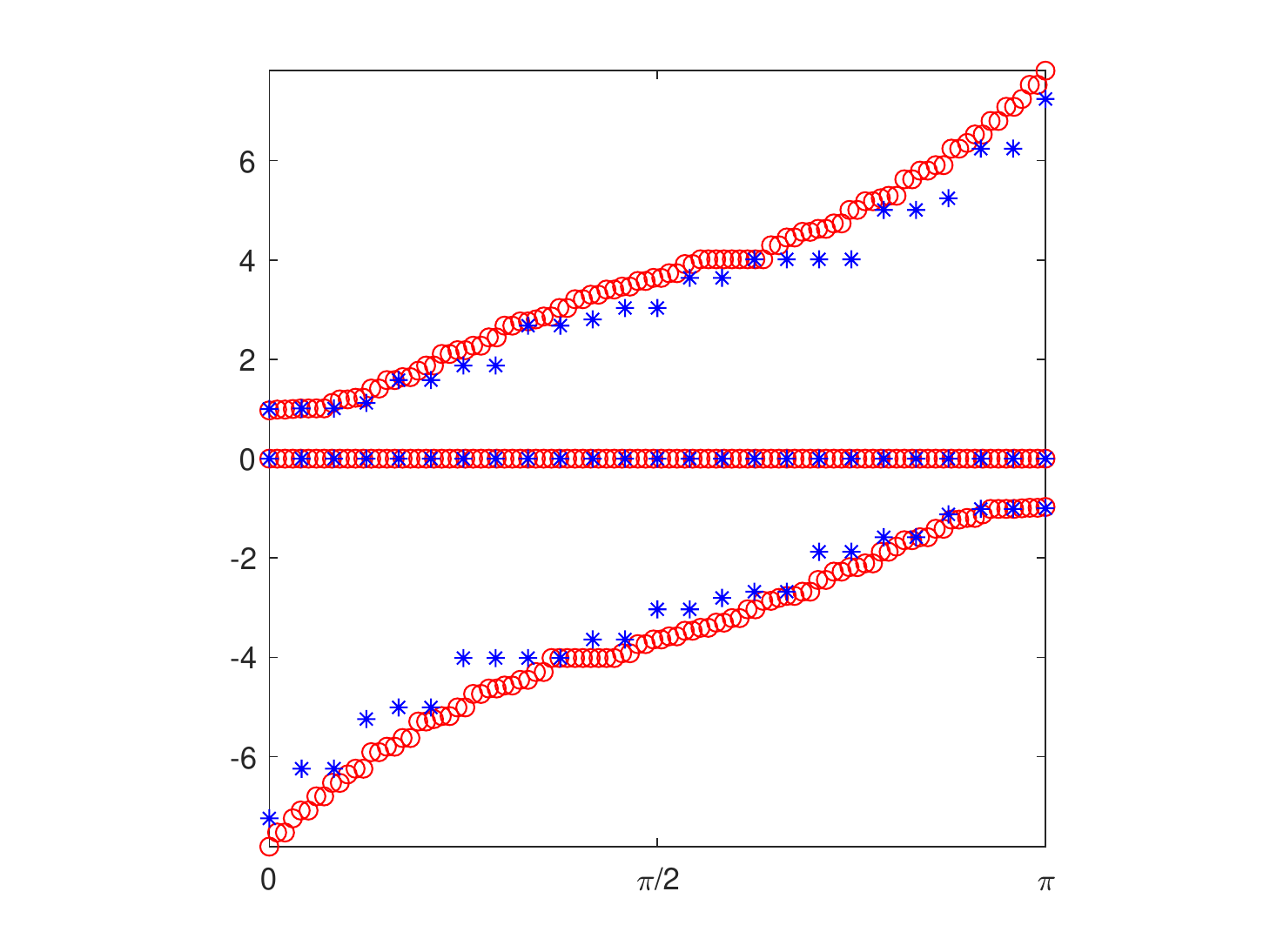}\label{n=5_n=10}}
	\subfloat[$n=6$]{\includegraphics[width=.5\columnwidth]{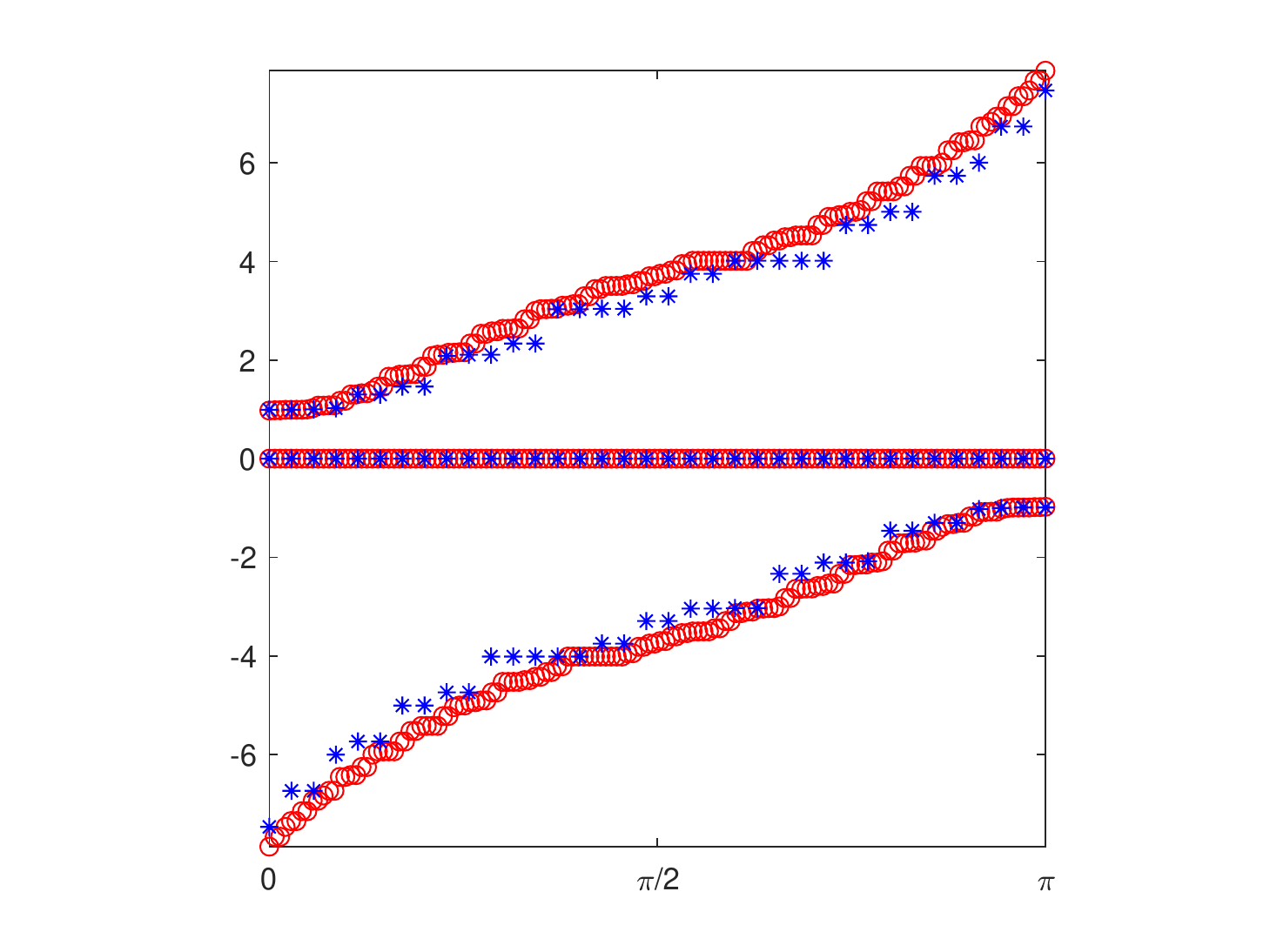}\label{n=6_n=12}}
	\caption{Comparison between the evaluation of the eigenvalue functions $\lambda^{(l)}(\mathbf{f})$, $l=1,2,3$, ordered in ascending way, on the grid $\boldsymbol{\theta}_{\bf n-e}$ contained in $P_l^{(n)}$ (${\color{red}\circ}$) and the corresponding evaluation on the grid twice as fine $\boldsymbol{\theta}_{\bf 2n-e}$ contained in $P_l^{(2n)}$ (${\color{blue}\ast}$). Each \lq curve' refers to a different value of $l$. The parameter $n$ equals $5$ and $6$ in subplots (a) and (b), respectively}\label{n=5_n=10andn=6_n=12}
\end{figure}
in which we compare the approximation of $\lambda^{(l)}(\textbf{f})$ on $\boldsymbol{\theta}_{\bf n}$, $n=5,6$ contained in $\mathbf{P}^{(n)}_l$ (ordered in ascending way) with the approximation of the same eigenvalue function on a grid that is twice as fine $\boldsymbol{\theta}_{\bf 2n-e}$, $n=5,6$ contained in $\mathbf{P}^{(2n)}_l$ (ordered in ascending way as well) for every $l=1,2,3$.

Then, for $n$  sufficiently large, if we order in ascending way $\mathbf{P}^{(n)}_l$, its extremes satisfy the following relations
\begin{equation*}
(\mathbf{P}^{(n)}_l)_1\approx m_l,\qquad (\mathbf{P}^{(n)}_l)_{n^2}\approx M_l, \quad l=1,2,3,
\end{equation*}
and we can can compute a satisfactory approximation of the $\{m_l,M_l\}_{l=1}^{3}$ from Proposition~\ref{pro:eigenbound}, e.g., by setting $n=3\cdot 10^3$, and $\alpha=\texttt{1.0e-04}$, we obtain the following approximations
\begin{align*}
\{m_1,M_1\}&\approx \{-8.006939205138657,-0.971179393341684\},\\
\{m_2,M_2\}&\approx \{0,0.00006086664699\},\\
\{m_3,M_3\}&\approx \{0.971268643759555,8.006939262908668\}.
\end{align*}
This clearly matches with the fact that the matrix--valued symbol is analytically singular in $(0,0)$, i.e.,
\begin{equation*}
\textbf{f}(0,0)=\begin{bmatrix}
0 & 0 & 0\\
0 & \alpha& -1\\
0&-1 & 0
\end{bmatrix},
\end{equation*}
hence $m_2=0$, nevertheless we stress again that this is not in contradiction with the fact that  $\mathcal{A}_N$ is non singular.

In conclusion, we can exploit Remark~\ref{rem_distr}, to provide an answer to \emph{Q2} determining how many eigenvalues are asymptotically contained in each of the three blocks. According to the relations~\eqref{eig_distr_piu}, \eqref{M_3} we expect the eigenvalues of $B_N$ to verify
\begin{equation} \label{rel_aut}
\begin{split}
\# \left\{ i \, : \, \lambda_i(B_N) \in (m1,M_1]\right\}&=\frac{3n^2}{3}+o(3n^2),\\
\# \left\{ i \, : \, \lambda_i(B_N) \in (m_2,M_2]\right\}&=\frac{3n^2}{3}+o(3n^2),\\
\# \left\{ i \, : \, \lambda_i(B_N) \in [m_3,M_3)\right\}&=\frac{3n^2}{3}+o(3n^2).
\end{split}
\end{equation}
and then to identify $3$ blocks
\begin{align*}
{\rm Bl}_1&=\left[\lambda_1 (B_N), \ldots,\lambda_{n^2}(B_N)\right],\\
{\rm Bl}_2&=\left[\lambda_{n^2+1} (B_N), \ldots,\lambda_{2n^2}(B_N)\right],\\
{\rm Bl}_3&=\left[\lambda_{2n^2+1} (B_N), \ldots,\lambda_{3n^2}(B_N)\right].
\end{align*}
Correspondingly, we can split the vector $\mathbf{P}^{(n)}$ containing the sampling of the eigenvalue functions on $\boldsymbol{\theta}_{\bf n-e}$ as follows
\begin{align*}
{\rm Eval}_1&= [(\mathbf{P}^{(n)})_{1}, \ldots, (\mathbf{P}^{(n)})_{n^2}],\\
{\rm Eval}_2&= [(\mathbf{P}^{(n)})_{n^2+1},\ldots, (\mathbf{P}^{(n)})_{2n^2}],\\
{\rm Eval}_3&= [(\mathbf{P}^{(n)})_{2n^2+1}, \ldots, (\mathbf{P}^{(n)})_{3n^2}].
\end{align*}
We stress again that~\eqref{rel_aut} allows for a number of outliers that is infinitesimal in the dimension $N$.

For example, for ${\bf n}=(n,n)=(40,40)$ ($N=4800 $), approximately $\frac{3n^2}{3}=1600$ eigenvalues should be in each block, by a straightforward numerical check one obtains
\begin{align}\label{rel_aut80}
\begin{split}
\# \left\{ i \, : \, \lambda_i(B_N) \in (m1_,M_1]\right\}=1600,\\
\# \left\{ i \, : \, \lambda_i(B_N) \in (m_2,M_2]\right\}=1421,\\
\# \left\{ i \, : \, \lambda_i(B_N) \in [m_3,M_3)\right\}=1600.
\end{split}
\end{align}
Therefore, we expect from that a certain number of eigenvalues of $B_N$ are in none of the blocks; in the example the effective $1421$ eigenvalues against the expected $1600$ in the second block. This is confirmed again by  Figure~\ref{autovalori_3blocchi} in which we highlight represent in blue the whole spectrum of $B_N$ and highlight in black the outliers not belonging to the blocks.
\begin{figure}[htb]
	\centering
	\includegraphics[width=0.65\columnwidth]{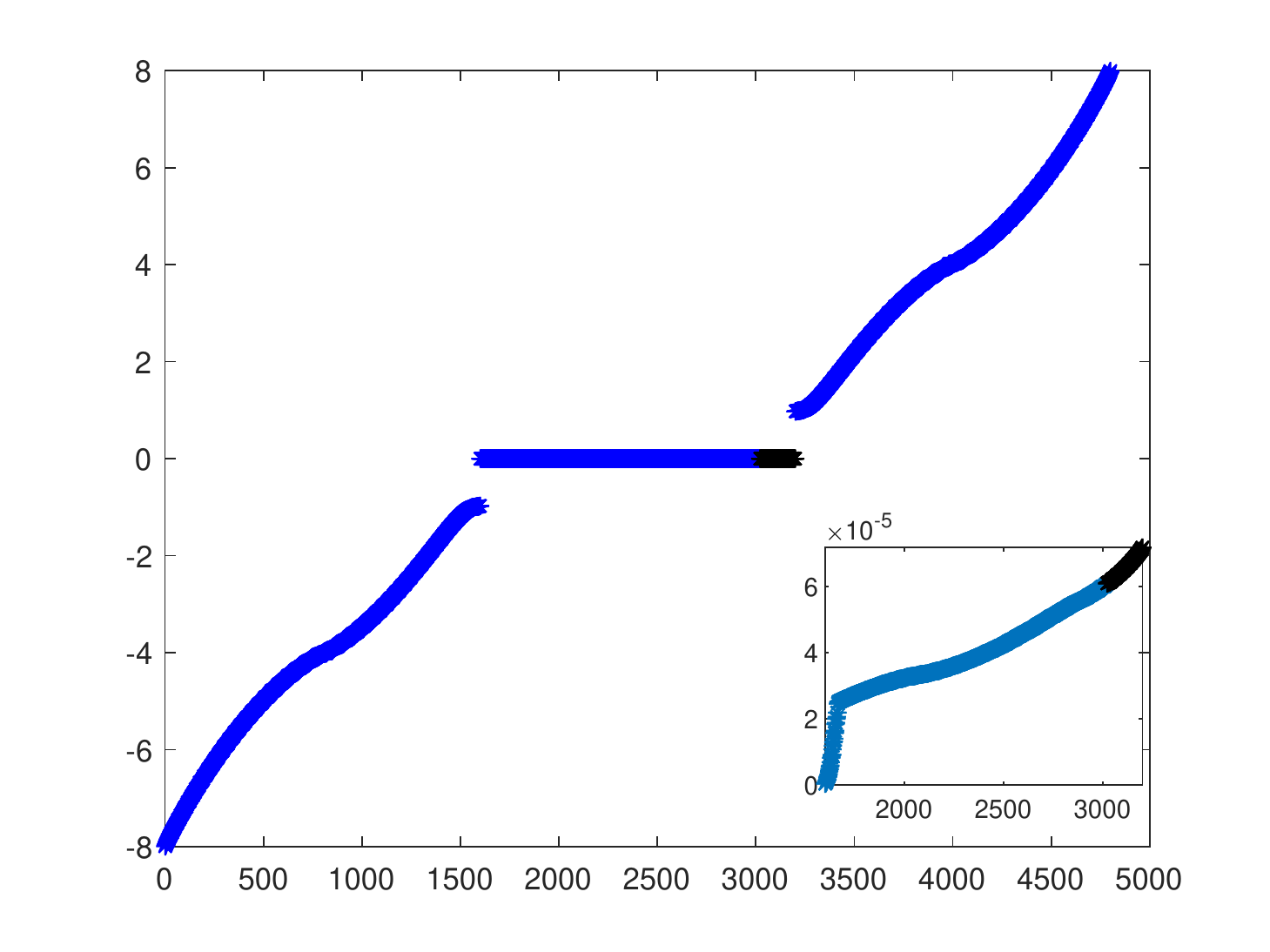}
	\caption{Eigenvalues of $B_N$ for ${\bf n}=(n,n)=(40,40)$ ({\color{black}$\ast$}) together with the eigenvalues of $B_N$ satisfying one of the relations \eqref{rel_aut} ({\color{blue}$\ast$}), for $\alpha=\texttt{1.0e-04}$}\label{autovalori_3blocchi}
\end{figure}
On the other hand, such a phenomenon is in line with \eqref{rel_aut}, since the order of what is missing/exceeding is infinitesimal in the dimension $N$. As an example, in Table~\ref{tab:outliers_m2_M2} we compare the actual number of eigenvalues of $B_N$ contained in the second interval $(m_2,M_2]$ with the expected number $n^2$. In such way, we succeed in counting the outliers of $B_N$ in $(m_2,M_2]$, whose cardinality behaves as $O(\sqrt{3n^2})$.
\begin{table}[htb]
	\centering
	\begin{tabular}{*5{c}}
		\toprule
		$n$&$\#\{\lambda \in (m_2,M_2]\}$& $n^2$& $\#\{\lambda \notin (m_2,M_2]\}$ &$\#\{\lambda \notin (m_2,M_2]\}/\sqrt{3n^2}$\\
		\midrule
		10 & 74&100&26&$0.086$\\
		20& 353 &400&47&$0.039$\\
		40& 1421&1600&179&$0.037$\\
		80& 5694&6400&706&$0.036$\\
		\bottomrule
	\end{tabular}
	\caption{Comparison of the effective number of eigenvalues of $B_N$ contained in the second interval $(m_2,M_2]$ with the expected number $n^2$}\label{tab:outliers_m2_M2}
\end{table}
{A further and more natural evidence of relation~\eqref{eig_distr_piu} can be obtained by comparing block by block the eigenvalues of $B_N$ with the sampling of the eigenvalue functions of $\textbf{f}$, that is comparing Bl$_1$, Bl$_2$, Bl$_3$,  with Eval$_1$, Eval$_2$, Eval$_3$, respectively. Indeed we want to compare the eigenvalues of $B_N$ (properly ordered) with the evaluation of $\lambda^{(l)}(\mathbf{f})$ $l=1,2,3$ at $\boldsymbol{\theta}_{\bf n-e}$, using the values that are present in the blocks of~$\mathbf{P}^{(n)}$.}

More precisely, we compare the elements of Eval$_t$  
with the elements of Bl$_t$ by means of the following  {matching algorithm}:
\begin{itemize}
	\item save the couples $(\theta_{n-1}^{(j_t)},\theta_{n-1}^{(k_t)})$ of $\boldsymbol{\theta}_{\bf n-e}$ to which the elements of Eval$_t$ are associated with;
	{
		\item for a fixed $\lambda\in {\rm Bl}_t$ find $\tilde \eta\in{\rm Eval}_t$ such that
		\begin{equation*}
		\tilde \eta= \arg\min_{\eta\in{\rm Eval}_t}\|\lambda-\eta\|;
		\end{equation*}
		\item associate $\lambda$ to the couple $(\theta_{n-1}^{(j_t)},\theta_{n-1}^{(k_t)})$ corresponding to $\tilde \eta$.
	}
\end{itemize}
Making use of the previous algorithm, in Figure~\ref{approx_2d}, we compare the eigenvalues of $B_N$ with $\lambda^{(l)}(\mathbf{f})$, $l=1,2,3$ displayed as a mesh on $\boldsymbol{\theta}_{\bf n-e}$, for $n=40$.
The eigenvalues of $B_N$ mimic, up to some outliers shown in the Figure~\ref{approssimazioni_lambda2_n=40}, the sampling of the eigenvalue functions,  {numerically confirming the result given in Theorem~\ref{thm:spectral_magic}.}

\begin{figure}[htb]
	\subfloat[$\lambda^{(1)}(\mathbf{f})$]{
		\includegraphics[width=.30\columnwidth]{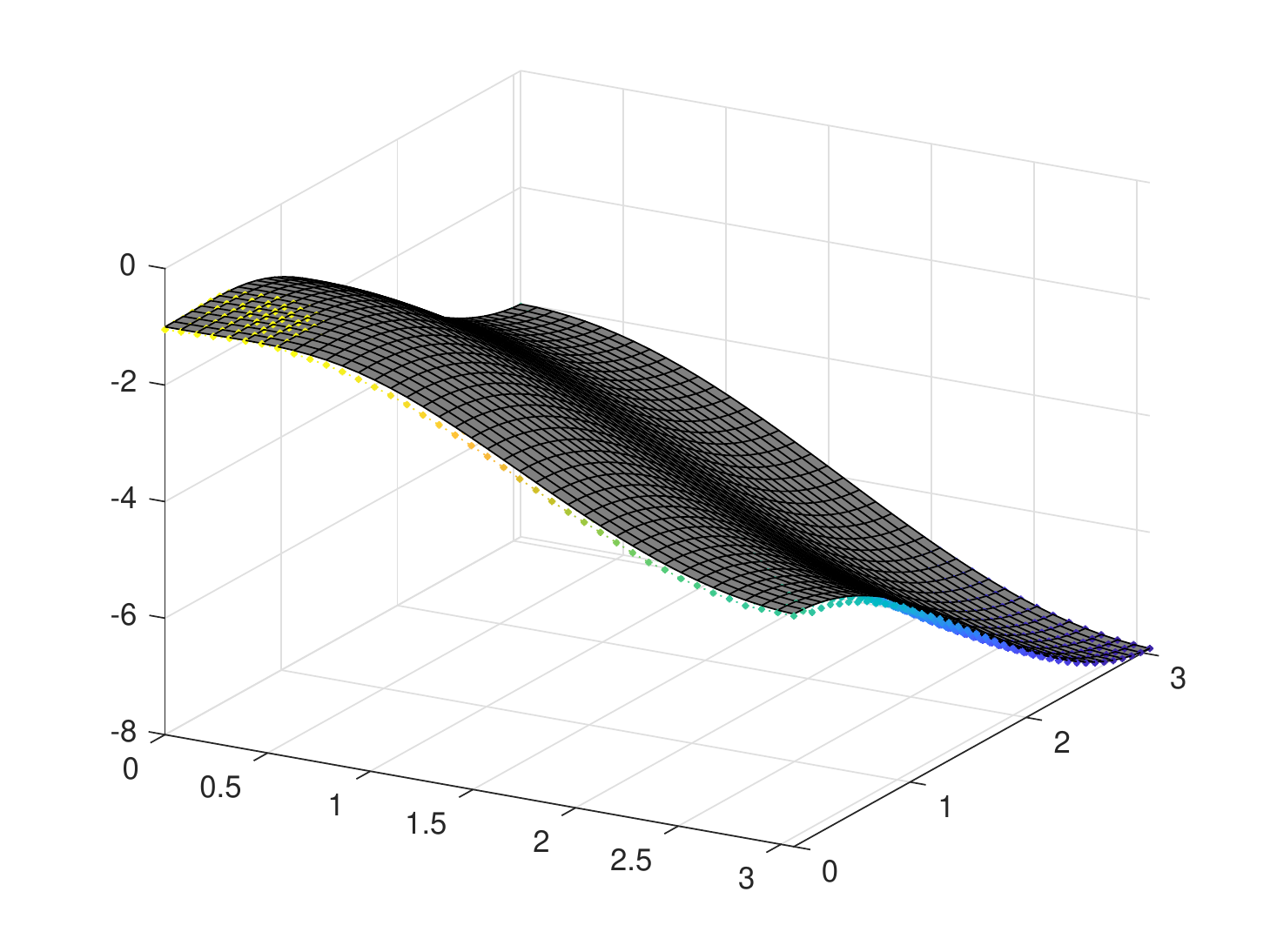}
		\label{approssimazioni_lambda1_n=40}}
	\hfill
	\subfloat[$\lambda^{(2)}(\mathbf{f})$]{
		\includegraphics[width=.30\columnwidth]{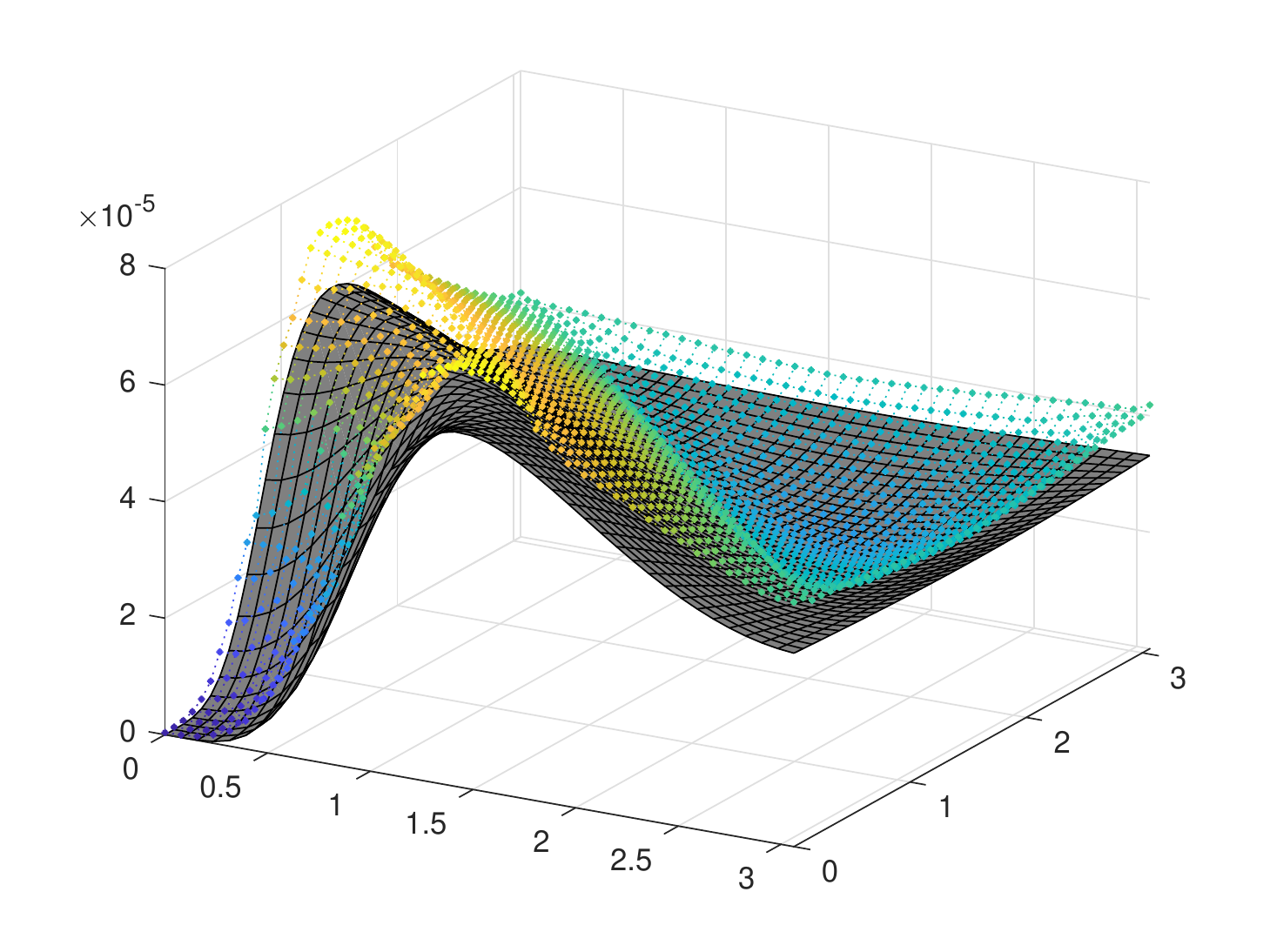}
		\label{approssimazioni_lambda2_n=40}}
	\hfill
	\subfloat[$\lambda^{(3)}(\mathbf{f})$]{
		\includegraphics[width=.30\columnwidth]{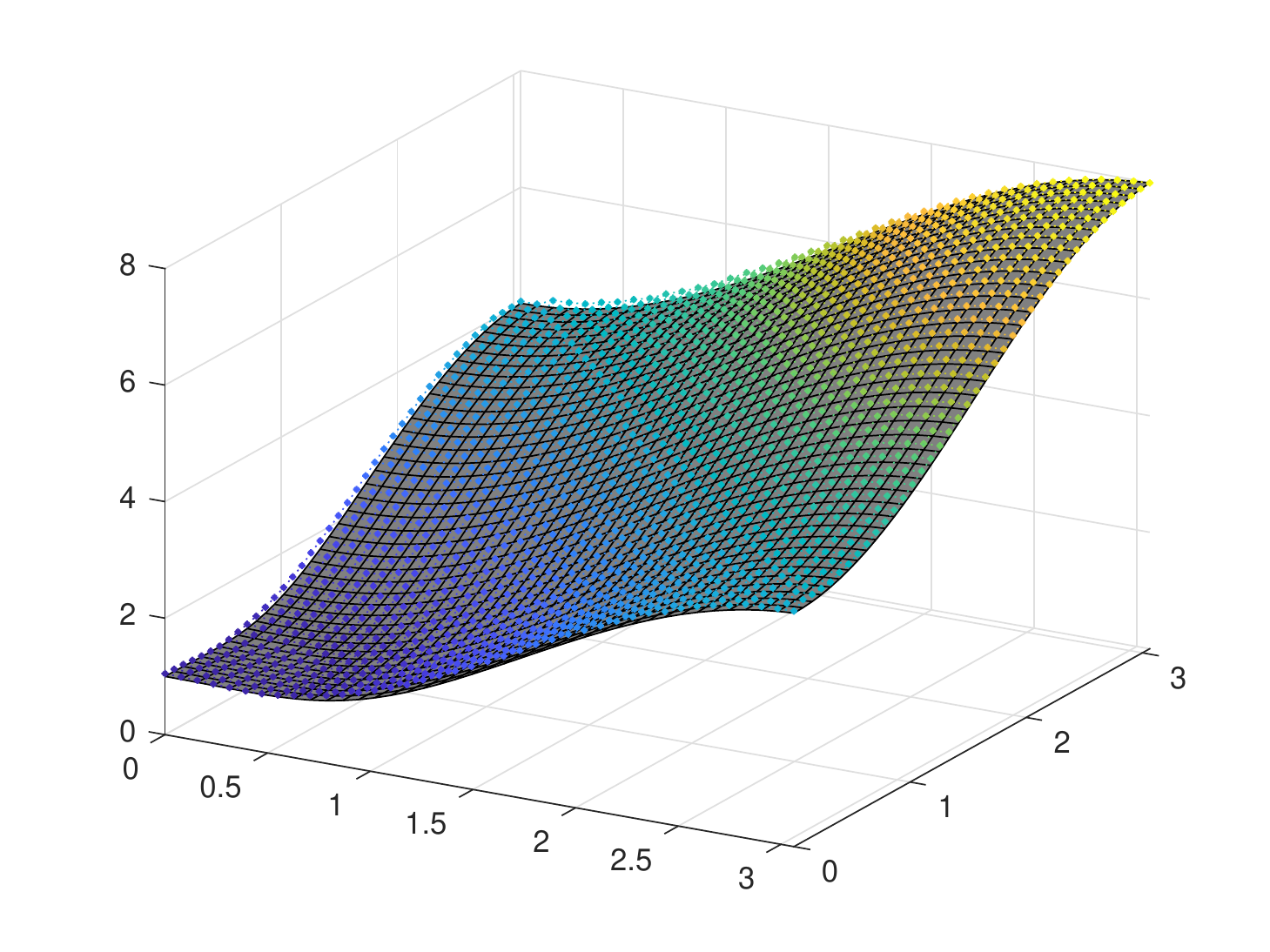}
		\label{approssimazioni_lambda3_n=40}}
	\caption{Comparison between the eigenvalues of $B_N$ and $\lambda^{(l)}(\textbf{f})$, $l=1,2,3$ displayed as a mesh on $\boldsymbol{\theta}_{\bf n-e}$, when $n=40$}\label{approx_2d}
\end{figure}

\subsection{From Poisson to advection-diffusion equations}
\label{sec:Study_BN_perm_non_herm}
We have built the whole construction using as constraint the Poisson differential equation, this is not restrictive since the analysis
can be transparently extended to encompass constraints given by a generic elliptic differential equations, i.e.,
\begin{equation}\label{eq:adevection_diffusion_reaction}
- \nabla^2 y +  \mathbf{c}\cdot\nabla y + r y = z.
\end{equation}
The matrix sequence~\eqref{eq:the_linear_system} maintains the same $3\times3$ block structure, but with a different (1,3) and (3,1) block $\bar{Z}$. The latter,
whenever $\mathbf{c} = (c_1,c_2) \neq 0$, is no more symmetric since the new constraint is no more self--adjoint. Specifically, the new block $\bar{Z}$ can be decomposed into the sum of three terms,
\begin{equation*}
\bar{Z} = \bar{K} + \bar{V} + \gamma \bar{M}, \qquad (\bar{V})_{i,j} = \int_{\tau_h} (\mathbf{c} \cdot \nabla \phi_i) \phi_j d\mathbf{x},
\end{equation*}
with  $V \neq V^T$. Therefore, the relative scaled version is given by
\begin{equation}
\mathcal{S}_N = \mathcal{D}^{(1)}_N \bar{\mathcal{S}}_N \mathcal{D}^{(2)}_N = \begin{bmatrix}
h^4  {M} & O &  {Z}^T \\
O & \alpha  {M}  & - {M} \\
{Z} & - {M} & O \\
\end{bmatrix}, \quad Z = K + h V + h^2 M.\label{eq:spectrally_analyzed_matrix_non_hermitian}
\end{equation}
By means of a GLT perturbation argument from Section~\ref{sec:background_and_definitions},  {and exploiting the analysis in \cite[Section~7.4]{SerraLibro2} for the presence of lower order differential terms}, we can obtain again a characterization {of the eigenvalues of $\mathcal{S}_N$ in~\eqref{eq:spectrally_analyzed_matrix_non_hermitian}} that is analogous to the one we gave in Theorem~\ref{thm:spectral_magic}. 

\begin{Proposition}\label{pro:non-hermitian-case}
	The matrix sequence $\{\mathcal{S}_N\}_N$ from~\eqref{eq:spectrally_analyzed_matrix_non_hermitian} is distributed in the eigenvalue sense as the matrix--valued function $\mathbf{f}$ from Theorem~\ref{thm:spectral_magic}.
\end{Proposition}

\begin{proof}
	Follows from Theorem~\ref{thm:spectral_magic}, the techniques adopted in its proof, and from \textbf{GLT5} applied to $\mathcal{S}_N = \mathcal{A}_N + \mathcal{Y}_N$, where
	\begin{equation*}
	\mathcal{Y}_N = \begin{bmatrix}
	O & O & hV^T + h^2M \\
	O & O & O \\
	hV + h^2M & O & O 
		\end{bmatrix}.
	\end{equation*}
\end{proof}

\section{An optimal preconditioning strategy}
\label{sec:efficientsolution}
In this section we analyze an effective procedure to precondition the GMRES method for the solution of the systems~\eqref{eq:spectrally_analyzed_matrix}, and~\eqref{eq:spectrally_analyzed_matrix_non_hermitian}. There exist indeed many preconditioners for the linear systems of saddle--point type exploiting their block structure, see, e.g, the review~\cite{benzi2005numerical} the comparisons in~\cite{MR3564863}, and, more specifically, the approaches described in~\cite{MR2785775,MR3187670,MR2759604,MR3787487}. What we present here belongs to this class, and is built with the objective of obtaining  {algorithmic scalability}, i.e.,
independence of the number of iteration from $h$, and optimality with respect to the parameter $\alpha$, i.e., independence {of} the number of iteration also with
respect to it. To achieve this kind of results the classical techniques can be broadly divided into three classes, the case of definite Hermitian preconditioners for which it is possible to retrieve a cluster of the eigenvalue sense from a cluster of the singular values~\cite{MR2176808,MR3187670,MR2785775}, that allows also for the use of the MINRES method; the case of the indefinite Hermitian preconditioners, and non Hermitian preconditioner~\cite{MR2759604,MR3787487}. We focus here on the last approach, while benefiting both from the spectral distribution of the sequence~$\{T_{\bf n}(m)\}_{\bf n}$ and $\{T_{\bf n}(\kappa)\}_{\bf n}$ of the Sections~\ref{Study_BN_perm}, \ref{sec:Study_BN_perm_non_herm}, and from the block form of the matrices~$\mathcal{A}_N$ and $\mathcal{S}_N$. Specifically, we propose the following preconditioner
\begin{equation}\label{eq:finally_permutation_for_the_preconditioner}
\mathcal{P}_N \begin{bmatrix}
{\mathbf{z}}_1\\
{\mathbf{z}}_2\\
{\mathbf{z}}_3\\
\end{bmatrix} = \begin{bmatrix}
O & \alpha K^T & O \\
O & \alpha M & -M \\
K & -M & O
\end{bmatrix}\begin{bmatrix}
{\mathbf{z}}_1\\
{\mathbf{z}}_2\\
{\mathbf{z}}_3\\
\end{bmatrix} = 
\begin{bmatrix}
{\mathbf{r}}_1\\
{\mathbf{r}}_2\\
{\mathbf{r}}_3\\
\end{bmatrix}.
\end{equation}
This is clearly an indefinite, and non Hermitian matrix, nevertheless, the linear systems involving it can be easily solved by the following back--substitution procedure:
\begin{enumerate}
	\item Solve $\alpha K^T{\mathbf{z}}_2=\mathbf{r}_1$;
	\item Solve $M {\mathbf{z}}_3 =\alpha M{\mathbf{z}}_2-{\mathbf{r}}_2$;
	\item Solve $K{\mathbf{z}}_1={\mathbf{r}}_3+M{\mathbf{z}}_2$.
\end{enumerate}
We stress that this does not require the approximation of any of the possible Schur complements of $\mathcal{A}_N$ ($\mathcal{S}_N$), thus greatly simplifying the construction of the preconditioner. Moreover, we are going to prove now that this choice provides a strong cluster at 1 for the eigenvalues of the preconditioned linear system while obtaining also the independence {of} $\alpha$. We obtain this result in two steps by means of the GLT theory showing that the matrix sequence $\{\mathcal{P}_N^{-1}\mathcal{A}_N\}_N$ is distributed in the sense of the eigenvalues as $\textbf{1}$. First, in Proposition~\ref{prop:from_preconditioned_to_eigenvalue_problem}, we show that the eigenvalues of the preconditioned matrix $\mathcal{P}_N {^{-1}}\mathcal{A}_N$ are either $1$, or the generalized eigenvalues of an auxiliary problem, then, in Lemma~\ref{lemma:disribuzione_eig_problem}, we prove that the matrix sequence associated to the latter is indeed distributed in the eigenvalue sense as the function $\textbf{1}$, thus obtaining that the eigenvalues of the preconditioned system are strictly clustered at $1$.
\begin{Proposition}\label{prop:from_preconditioned_to_eigenvalue_problem}
	Let $\mathcal{A}_N$ ($\mathcal{S}_N$) be the coefficient matrix in~\eqref{eq:spectrally_analyzed_matrix} (respectively in~\eqref{eq:spectrally_analyzed_matrix_non_hermitian}), and let $\mathcal{P}_N$ be the associated preconditioner from~\eqref{eq:finally_permutation_for_the_preconditioner}.
	Then, the eigenvalues of the preconditioned matrix $\mathcal{P}_N^{-1}\mathcal{A}_N$ are 
	\begin{itemize}
		\item $\lambda_j = 1$ for $j=1,\ldots,2N(\mathbf{n})$,
		\item $\lambda_j$ for $j=2N(\mathbf{n})+1,\ldots,N(3,\mathbf{n})$ given by the solution of the generalized eigenvalue problem \[\left(\frac{h^4}{\alpha}M+K^{T}M^{-1}K\right)\textbf{x}_1={\lambda}K^{T}M^{-1}K\textbf{x}_1,\]
		with $\textbf{x}_1\neq\textbf{0}\in  {\mathbb{R}^{N(\mathbf{n})}}.  $	
	\end{itemize}
\end{Proposition}
\begin{proof}
	For each $n$,  $\lambda$ is an eigenvalue of the matrix 
	$\mathcal{P}^{-1}_N\mathcal{A}_N$ if $(\lambda,\bf{x})$ is an eigenpair of the eigenvalue  problem  \[\mathcal{A}_N{\bf x}=\lambda \mathcal{P}_N {\bf x},\]
	with\begin{equation*}
	{\bf x}=\begin{bmatrix}
	\textbf{x}_1\\
	\textbf{x}_2\\
	\textbf{x}_3
	\end{bmatrix}\neq{\bf 0}\in  {\mathbb{R}^{N(3,\mathbf{n})}}.
	\end{equation*}
	That is $(\lambda,\bf{x})$ is solution of
	\begin{equation*}
	\begin{bmatrix}
	h^4M & O & K^T \\
	O & \alpha M & -M \\
	K & -M & O
	\end{bmatrix}\begin{bmatrix}
	\textbf{x}_1\\
	\textbf{x}_2\\
	\textbf{x}_3
	\end{bmatrix}=\lambda \begin{bmatrix}
	O & \alpha K^T & O \\
	O & \alpha M & -M \\
	K & -M & O
	\end{bmatrix} \begin{bmatrix}
	\textbf{x}_1\\
	\textbf{x}_2\\
	\textbf{x}_3
	\end{bmatrix}.
	\end{equation*}
	It is clear from the second and the third ``block'' equations that $(1,{\bf x})$ is an eigenpair for the latter problem for all  {the vectors in the $N(2,\mathbf{n})$ subspace of $\mathbb{R}^{N(3,\mathbf{n})}$} 
	\[\left\lbrace{\bf x}=\begin{bmatrix}
	\textbf{x}_1\\
	\textbf{x}_2\\
	\textbf{x}_3
	\end{bmatrix} {\rm s.t. }\;\; \textbf{x}_3= \alpha\textbf{x}_2-h^4K^{-T}M\textbf{x}_1, \quad \forall\,\textbf{x}_1,\textbf{x}_2\in  {\mathbb{R}^{N(\mathbf{n})}}\right\rbrace.\]
	Otherwise, if $\lambda\neq 1$, from the third ``block'' equation
	\[(1-\lambda)K \textbf{x}_1=(1-\lambda)M \textbf{x}_2,\]
	follows
	\[\textbf{x}_2=M^{-1}K\textbf{x}_1.\]
	And thus, by substitution, we easily find
	\[\textbf{x}_3=\alpha M^{-1}K\textbf{x}_1\]
	and thus the remaining eigenpairs are given by the solution of
	\[\left(\frac{h^4}{\alpha}M+K^{T}M^{-1}K\right)\textbf{x}_1={\lambda}K^{T}M^{-1}K\textbf{x}_1.\]
\end{proof}
\begin{Lemma}\label{lemma:disribuzione_eig_problem}
	The matrix sequence  \[\Biggl\{\left(K^TM^{-1}K\right)^{-1}\left(\frac{h^4}{\alpha}M+K^T{M}^{-1}K\right)\Biggr\}_{\bf n},\]  associated to the generalized eigenvalue problem
	\[\left(\frac{h^4}{\alpha}M+K^TM^{-1}K\right)\textbf{x}_1=\lambda K^TM^{-1}K\textbf{x}_1,  \]	
	is distributed in the eigenvalue sense as ${\bf 1}$ over  $\mathcal{I}_2^+$.
\end{Lemma}
\begin{proof}
	The statement is equivalent to \[\Biggl\{(T_{\bf n}(\kappa)T^{-1}_{\bf n}(m)T_{\bf n}(\kappa))^{-1}\left(\frac{h^4}{\alpha}T_{\bf n}(m)+T_{\bf n}(\kappa)T_{\bf n}^{-1}(m)T_{\bf n}(\kappa)\right) \Biggr\}_{\bf n}\sim_{\lambda}({\bf 1},\mathcal{I}_2^+),\]  since, from  {\eqref{eq:m-bi-toeplitz} and~\eqref{eq:k-bi-toeplitz}}, we have that $M$ and $K$ are the symmetric and positive definite matrices $T_{\bf n}(m)$ and $T_{\bf n}(\kappa)$, respectively. 
	
	Moreover the sequence $\biggl\{\frac{h^4}{\alpha}  T_{\bf n}(m)\biggr\}_{\bf n}$ is 
	distribuited in the singular value sense as $0$ over  $\mathcal{I}_2^+$. Hence from property {\bf GLT4} plus properties {\bf GLT2-GLT3} we have that the following GLT results hold:
	\begin{equation*}
	\biggl\{\frac{h^4}{\alpha}  T_{\bf n}(m)\biggr\}_{\bf n}\sim_{GLT}{\bf 0}, 
	\end{equation*}
	and
	\begin{equation*}
	\begin{split}
	&\{T_{\bf n}(m)\}_{\bf n}\sim_{GLT}m, \quad
	\{T_{\bf n}(\kappa)\}_{\bf n}\sim_{GLT}\kappa,\\
	&\{T_{\bf n}^{-1}(m)\}_{\bf n}\sim_{GLT}\frac{1}{m}, \quad \{T_{\bf n}^{-1}(\kappa)\}_{\bf n}\sim_{GLT}\frac{1}{\kappa}.
	\end{split}
	\end{equation*}
	Exploiting again {\bf GLT 2--GLT4} we obtain that 
	\begin{equation*}
	\{T_{\bf n}(\kappa)T_{\bf n}^{-1}(m)T_{\bf n}(\kappa)\}_{\bf n}\sim_{GLT}\frac{m}{\kappa^2}
	\end{equation*}
	and
	\begin{equation*}
	\biggl\{\frac{h^4}{\alpha} T_{\bf n}(m)+T_{\bf n}(\kappa)T_{\bf n}^{-1}(m)T_{\bf n}(\kappa)\biggr\}_{\bf n}\sim_{GLT}\frac{\kappa^2}{m}.
	\end{equation*}
	Since the matrix $T_{\bf n}(\kappa)T_{\bf n}^{-1}(m)T_{\bf n}(\kappa)$ is positive definite, then Theorem~\ref{th:esercizio_libro} implies 
	\[\biggl\{(T_{\bf n}(\kappa)T^{-1}_{\bf n}(m)T_{\bf n}(\kappa))^{-1}\left(\frac{h^4}{\alpha} T_{\bf n}(m)+T_{\bf n}(\kappa)T_{\bf n}^{-1}(m)T_{\bf n}(\kappa)\right)\biggr\}_{\bf n}\sim_{GLT, \sigma,\lambda}({\bf 1},\mathcal{I}_2^+)\]
	and, hence, the thesis.
\end{proof}

{
	\begin{Remark}\label{remark:asymptotic_regime_of_generalized_eigenvalues}
		Let us stress that the conclusion in Lemma~\ref{lemma:disribuzione_eig_problem} is again an asymptotic result for $h \rightarrow 0$ that is then valid for a fixed value of the parameter $\alpha$. Furthermore, it permits also an answer to \emph{Q3} characterizing the condition number of the preconditioned matrix sequence. Specifically, if we let $X$ be the matrix of the generalized eigenvectors for the pencil $(K,M)$, i.e., if $X$ is an invertible matrix such that 
		\begin{equation*}
		KX = MXD, \text{ with } \begin{array}{l} X^T K X = \diag(d^{(K)}_\mathbf{1},\ldots,d^{(K)}_\mathbf{n}) \equiv D^{(K)}, \\ X^T M X = \diag(d^{(M)}_\mathbf{1},\ldots,d^{(M)}_\mathbf{n}) \equiv D^{(M)}, \\ D = \diag\left(\frac{d^{(K)}_\mathbf{1}}{d^{(M)}_\mathbf{1}},\ldots,\frac{d^{(K)}_\mathbf{n}}{d^{(M)}_\mathbf{n}}\right),
		\end{array}
		\end{equation*}
		then we find 
		\begin{equation*}
		\begin{split}
		\left(K^TM^{-1}K\right)^{-1} & \left(\frac{h^4}{\alpha}M+K^T{M}^{-1}K\right) =   \left( X^{-T} D^{(K)} X^{-1} X {D^{(M)}}^{-1} X^{T} X^{-T} D^{(K)} X^{-1} \right)^{-1}\\& 
		\left( \frac{h^4}{\alpha} X^{-T} D^{(M)} X^{-1} + X^{-T} D^{(K)} X^{-1} X {D^{(M)}}^{-1} X^{T} X^{-T} D^{(K)} X^{-1} \right) \\
		= &\; X D^{(M)} (D^{(K)})^{-2} \left( \frac{h^4}{\alpha} D^{(M)} + (D^{(K)})^2 (D^{(M)})^{-1} \right) X^{-1}.
		\end{split}
		\end{equation*}
		It is then straightforward to use~\eqref{eq:m-bi-toeplitz} and~\eqref{eq:k-bi-toeplitz} to estimate the maximum eigenvalues of the generalized eigenvalue problem in Proposition~\ref{prop:from_preconditioned_to_eigenvalue_problem} as an $O(\alpha^{-1})$. This means that the asymptotic regime described in Lemma~\ref{lemma:disribuzione_eig_problem} is evident whenever $h^4$ becomes smaller than the fixed value of $\alpha$ of the given problem.  
	\end{Remark}
}

{
	We can now answer to question \emph{Q3} for both the matrix sequences $\{\mathcal{P}_N^{-1}\mathcal{A}_N\}_N $, and $\{\mathcal{P}_N^{-1}S_N\}_N$ of the Subsection \ref{sec:Study_BN_perm_non_herm}, where in the definition of the preconditioner~\eqref{eq:finally_permutation_for_the_preconditioner} $Z$ plays the same role of $K$.}

\begin{Theorem}\label{th:collect}
	The matrix sequences $\{\mathcal{P}_N^{-1}\mathcal{A}_N\}_N \sim_\lambda (\mathbf{1},\mathcal{I}_2^+)$, $\{\mathcal{P}_N^{-1}\mathcal{S}_N\}_N \sim_\lambda (\mathbf{1},\mathcal{I}_2^+)$ independently {of} $\alpha$.
\end{Theorem}

Moreover, an analogous spectral result to Theorem \ref{th:collect} can be given for the sequence $\{\mathcal{P}_{\text{BCT}}^{-1}\mathcal{A}_N\}_{N}$ (respectively, $\{\mathcal{P}_{\text{BCT}}^{-1}\mathcal{S}_N\}_{N}$), for
\begin{equation*}
\mathcal{P}_{\text{BCT}} = \begin{bmatrix}
O & O &  {K}^T \\
O & \alpha  {M}  & - {M} \\
{K} & - {M} & O \\
\end{bmatrix}.
\end{equation*}
\begin{Theorem}\label{thm:clusteringasymptotic}
	The matrix sequences $\{\mathcal{P}_{\text{BCT}}^{-1}\mathcal{A}_N\}_N \sim_\lambda (\mathbf{1},\mathcal{I}_2^+)$, $\{\mathcal{P}_{\text{BCT}}^{-1}\mathcal{S}_N\}_N \sim_\lambda (\mathbf{1},\mathcal{I}_2^+)$ independently {of}~$\alpha$.
\end{Theorem}
\begin{proof}
	The proof follows the proofs of the Proposition \ref{prop:from_preconditioned_to_eigenvalue_problem} and Lemma \ref{lemma:disribuzione_eig_problem}, replacing the expression of $\mathcal{P}_N$ with that of $\mathcal{P}_{\text{BCT}}$.
\end{proof}
This is indeed an example of a block--counter--triangular preconditioner in the style of~\cite{MR2785775}. 

\begin{Remark}\label{rmk:alternative_preconditioners}
{The preconditioner proposed in~\cite{MR2785775} takes the lower anti--triangular part of a different permutation of the system matrix $\mathcal{A}_N$, and considers also a different scaling. By this approach, the term that is dropped out in the preconditioner is not a correction of ``small'' norm, and this makes a substantial difference in the performances of the two approaches. Specifically, }comparing the results of Proposition~\ref{prop:from_preconditioned_to_eigenvalue_problem}, with~\cite[Theorem~3.1]{MR2785775}, it is straightforward to observe that in the latter case it is not possible to infer a cluster of the eigenvalues of the preconditioned system, specifically, for the rearranged system
\begin{equation*}
\tilde{\mathcal{P}}_{\text{BCT}}^{-1} \tilde{\mathcal{A}}_N = \begin{bmatrix}
O & O & -M \\
O & h^4 M & K^T \\
-M & K & O
\end{bmatrix}^{-1} \begin{bmatrix}
\alpha M & O & -M \\
O & h^4 M & K^T \\
-M & K & O
\end{bmatrix}.
\end{equation*}
The non-unit eigenvalues are the one of the matrix sequence $\{I + \alpha h^{-4} M^{-1} K M^{-1} K^T\}_N$, for which the clustering at one cannot be concluded. Similar observation can be made also for the null--space based block anti--triangular preconditioners~\cite{MR3187670} arising from the block anti--triangular factorization of the saddle--point matrix. {Furthermore, one could consider the preconditioner which neglects the (3,2) block of $\mathcal{\bar{A}}_N$, avoiding the reordering and the scaling. This would bring to the case where the non-unit eigenvalues are the solution of the following generalized eigenvalue problem
 \begin{equation*}
\begin{bmatrix}
h^2 M & O & K^T \\
O & \alpha h^2 M & -h^2 M \\
K & - h^2 M & O
\end{bmatrix} \begin{bmatrix}
\mathbf{x}_1 \\ \mathbf{x}_2 \\ \mathbf{x}_3
\end{bmatrix} = \lambda \begin{bmatrix}
h^2 M & O & K^T \\
O & \alpha h^2 M & -h^2 M \\
K & O & O
\end{bmatrix} \begin{bmatrix}
\mathbf{x}_1 \\ \mathbf{x}_2 \\ \mathbf{x}_3
\end{bmatrix},
\end{equation*}
and, then, we have a behavior analogous to the case with preconditioner $\tilde{\mathcal{P}}_{\text{BCT}}$ (i.e., the absence of a provable cluster of the preconditioned sequence). Precisely, the non-unit eigenvalues are of the form $\lambda_i=1+\mu_i$, where, $\mu_i$ are the reciprocal of the  eigenvalues of the matrix sequence $\{\frac{\alpha}{h^4}  M^{-1} K M^{-1} K^T\}_N$.  }
\end{Remark}

\subsection{Approximate iterative solution of the auxiliary linear systems}
\label{sec:approximate_solution}

The application of the proposed preconditioners requires the solution of auxiliary linear systems with the matrices $K$ ,$K^T$, and $M$ or, respectively, $Z$, $Z^T$, and $M$  { obtained from~\eqref{eq:finally_permutation_for_the_preconditioner}}. In both cases we are dealing with very common linear systems for which there exist highly efficient and specific solvers, e.g., fast Poisson solvers, multigrid methods of geometric, and algebraic type, inner--outer Krylov solver with incomplete factorization preconditioner, and several combinations of all the previous. Potentially, any optimal preconditioner for these matrices could be included in the present framework without spoiling the overall construction, the actual choice is indeed a matter of computational framework; { see, e.g.,~\cite[Chapter~3.8]{MR3793630}. For the solution of the systems involving the mass matrix $M$ a straightforward solution is using the unpreconditioned CG method or its preconditioned version. In the latter case, we use either a modified incomplete Cholesky factorization with drop–tolerance \texttt{1e-2} or
a standard algebraic multigrid. We stress that the solution of the system involving the stiffness matrix
can be machine-dependent; see, e.g., Figure~\ref{fig:innersolver}.  We easily observe that the fastest solution
with the required accuracy for the system involving the $K = T_\mathbf{n}(k)$ is obtained by using the PCG with a standard AMG
preconditioner. On the other hand, for the non symmetric case we can use the BiCGstab method together with a modified incomplete LU factorization of Crout type.}
\begin{figure}[htbp]
	\centering
	\includegraphics[width=\columnwidth]{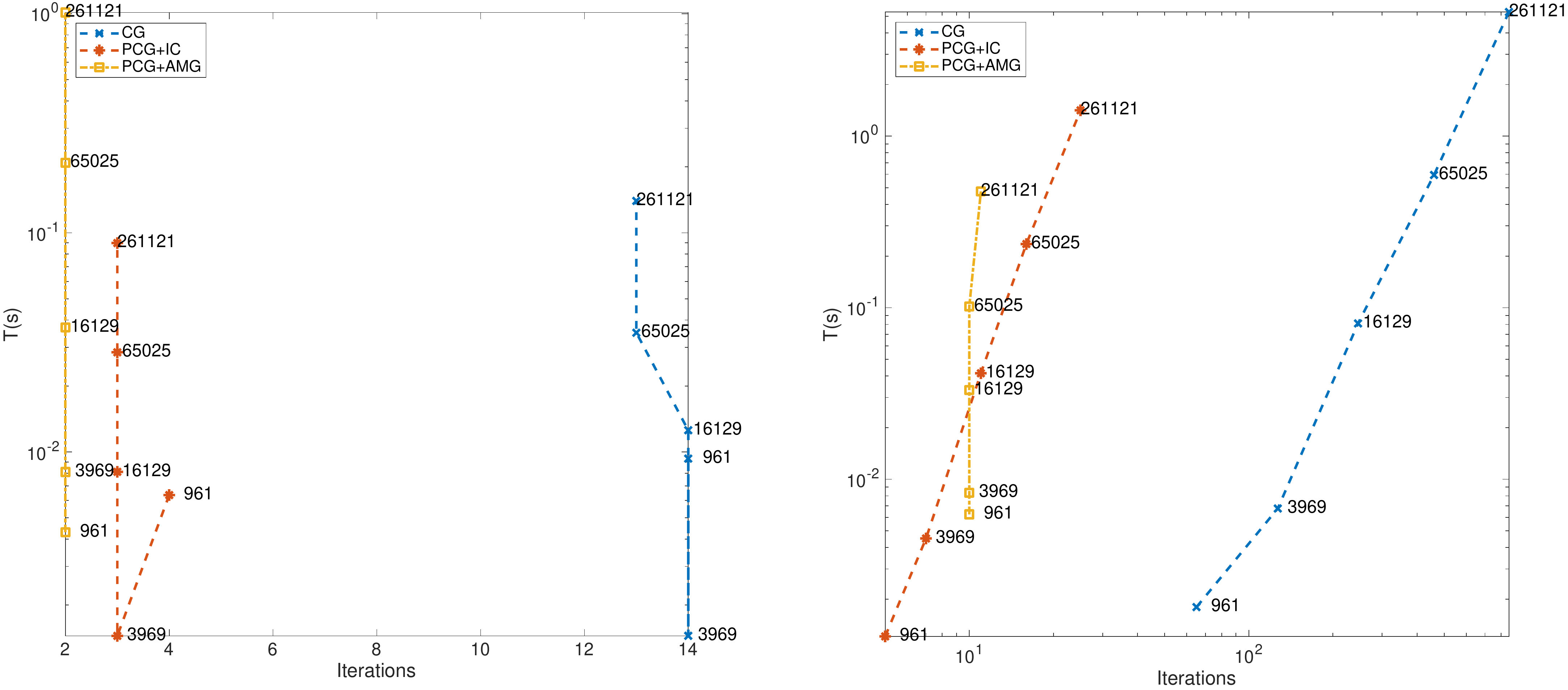}
	\caption{Comparison of solving routines for the auxiliary linear systems, on the left we compare the solution for the system involving the mass matrix $M$, while on the right the comparison is for the Hermitian stiffness matrix $K$. The comparisons do not take into account the building time for the various preconditioner since it is then distributed among the repeated solution.  {The maximum number of allowed PCG iterations is the size of the problem, while the stopping criterion {on the relative residual is set to a tolerance of~\texttt{1e-8}}}}\label{fig:innersolver}
\end{figure}
Nevertheless, as we discuss in the next Section~\ref{sec:numerical_experiments}, the time--efficiency in the auxiliary solve it is not so crucial, observe that already the direct method gives acceptable results under this aspect. What really matters is the combination of the achieved accuracy of the auxiliary solve with the presence, and the possible accumulation, of the $\alpha$ factor in the right--hand side of the auxiliary linear systems. This will cause for their solution by a direct method to return better performances for the lowest value of $\alpha$.

\section{Numerical Examples}
\label{sec:numerical_experiments}

In this section we test the application of the preconditioners analyzed in Section~\ref{sec:efficientsolution} on some test problems. All the numerical tests
are made on a laptop running Linux with 8~Gb
memory and CPU Intel\textsuperscript{\textregistered} Core\texttrademark\ i7--4710HQ CPU with clock 2.50~GHz and MATLAB version 9.4.0.813654 (R2018a). We recall again that all the relevant matrices and right--hand sides are generated by means of the FEniCS library (v.2018.1.0)~\cite{AlnaesBlechta2015a,LoggOlgaardEtAl2012a}; see again Section~\ref{sec:discretization} for the details.

We test the solution procedure with the un--restarted GMRES method set to achieve a tolerance on the residual of \texttt{tol = 1e-6}, and a maximum number of iteration \texttt{maxit = 100}, and measure the number of iterations, and the timings in second. As test problem we consider an instance of a Poisson control problem~\eqref{eq:theproblem}, and one with the diffusion--advection--reaction constraint from Section~\ref{sec:Study_BN_perm_non_herm}.

\paragraph{Poisson} The first test problem is an instance of the Poisson control problem~\eqref{eq:theproblem}, in which we want to obtain the desired state,
\[y_d(x_1,x_2) = -\sin(8\pi x_1)\sin(8\pi x_2)+\sin(\pi x_1)\sin(\pi x_2),\]
while using the forcing term 
\[z(x_1,x_2) = 2\pi^2\sin(\pi x_1)+\frac{1}{128\pi^2}\sin(8\pi x_1)\sin(8\pi x_2).\]
We test the solution for regularization parameter $\alpha = \texttt{1.0e-03}, \texttt{1.0e-06},\texttt{1.0e-09}$, and collect the results in Table~\ref{tab:poissonproblem}. The  {approximate} preconditioners are applied inside the Flexible--GMRES method as discussed in Section~\ref{sec:approximate_solution}.
\begin{table}[hbtp]
	\centering
	{\begin{tabular}{ll|llllll|llll}
			\toprule
			& & \multicolumn{6}{c}{GMRES} & \multicolumn{4}{c}{FGMRES+PCG+IC} \\
			& & \multicolumn{2}{c}{$I_N$} &  \multicolumn{2}{c}{$\mathcal{P}_N$} &  \multicolumn{2}{c}{$\mathcal{P}_{\text{BCT}}$} &  \multicolumn{2}{c}{$\mathcal{P}_N$} &  \multicolumn{2}{c}{$\mathcal{P}_{\text{BCT}}$} \\
			$\alpha$ & N & IT & T(s) & IT & T(s) & IT & T(s) & IT & T(s) & IT & T(s) \\
			\midrule
			1.0e-03 & 147 & $\dagger$ & - & 3 & 3.0e-03 & 3 & \textbf{2.5e-03} & 3 & 4.4e-03 & 3 & 4.5e-03 \\
			& 675 & $\dagger$ & - & 3 & 6.4e-03 & 3 & \textbf{3.7e-03} & 3 & 4.7e-03 & 3 & 4.6e-03 \\
			& 2883 & $\dagger$ & - & 3 & 1.0e-02 & 3 & 9.9e-03 & 3 & \textbf{7.3e-03} & 3 & \textbf{7.3e-03} \\
			& 11907 & $\dagger$ & - & 2 & 3.1e-02 & 2 & 3.0e-02 & 2 & \textbf{2.3e-02} & 2 & \textbf{2.3e-02} \\
			& 48387 & $\dagger$ & - & 2 & 2.1e-01 & 2 & 1.7e-01 & 2 & \textbf{1.5e-01} & 2 & \textbf{1.5e-01} \\
			& 195075 & $\dagger$ & - & 2 & 9.4e-01 & 2 & 9.0e-01 & 2 & \textbf{7.3e-01} & 2 & 7.4e-01 \\
			& 783363 & $\dagger$ & - & 1 & 2.1e+00 & 1 & \textbf{2.0e+00} & 1 & 2.9e+00 & 1 & 2.9e+00 \\
			\cmidrule{3-12}
			1.0e-06 & 147 & $\dagger$ & - & 15 & 4.3e-03 & 15 & 4.1e-03 & 15 & \textbf{1.5e-03} & 15 & \textbf{1.5e-03} \\
			& 675 & $\dagger$ & - & 14 & 1.3e-02 & 14 & \textbf{1.2e-02} & 14 & 1.7e-02 & 14 & 1.7e-02 \\
			& 2883 & $\dagger$ & - & 9 & 3.0e-02 & 9 & 3.0e-02 & 10 & \textbf{2.4e-02} & 10 & \textbf{2.4e-02} \\
			& 11907 & $\dagger$ & - & 6 & 1.0e-01 & 6 & 9.3e-02 & 6 & \textbf{7.0e-02} & 6 & 7.4e-02 \\
			& 48387 & $\dagger$ & - & 4 & 3.1e-01 & 4 & 3.0e-01 & 4 & 2.5e-01 & 4 & \textbf{2.1e-01} \\
			& 195075 & 86 & 3.8e+00 & 2 & 8.7e-01 & 2 & 8.4e-01 & 2 & 7.8e-01 & 2 & \textbf{7.6e-01} \\
			& 783363 & 80 & 3.0e+01 & 2 & \textbf{4.3e+00} & 2 & \textbf{4.3e+00} & 2 & 4.5e+00 & 2 & 4.6e+00 \\
			\cmidrule{3-12}
			1.0e-09 & 147 & $\dagger$ & - & 27 & 9.6e-03 & 27 & 8.7e-03 & 27 & \textbf{3.2e-03} & 27 & 3.4e-03 \\
			& 675 & $\dagger$ & - & 54 & 6.3e-02 & 54 & \textbf{5.8e-02} & 54 & 7.9e-02 & 54 & 7.9e-02 \\
			& 2883 & $\dagger$ & - & 52 & 2.0e-01 & 52 & 2.1e-01 & 52 & \textbf{1.6e-01} & 52 & \textbf{1.6e-01} \\
			& 11907 & $\dagger$ & - & 33 & 5.8e-01 & 33 & 6.1e-01 & 33 & \textbf{5.2e-01} & 33 & 5.5e-01 \\
			& 48387 & $\dagger$ & - & 20 & \textbf{1.5e+00} & 20 & \textbf{1.5e+00} & 43 & 4.8e+00 & 42 & 4.8e+00 \\
			& 195075 & 86 & \textbf{2.8e+00} & 33 & 1.3e+01 & 33 & 1.3e+01 & 37 & 3.0e+01 & 36 & 2.9e+01 \\
			& 783363 & 80 & 3.0e+01 & 33 & \textbf{2.5e+01} & 33 & \textbf{2.5e+01} & $\dagger$ & - & $\dagger$ & - \\
			\bottomrule
	\end{tabular}}
	\caption{Poisson Control Problem. We compare both the number of iterations, and the solution time for the various preconditioners. Best timings are highlighted in bold face. When the method fails to converge, i.e., the method reaches the maximum number of iterations, a $\dagger$ is reported.  {The inner tolerance for the PCG is set to 1e-8}}\label{tab:poissonproblem}
\end{table}
What we observe is that the approximate solution are at an advantage for the higher value of $\alpha$, while perform poorly for the smallest $\alpha = \texttt{1.0e-09}$. We stress that this effect is more  {connected} to the behavior of the accuracy in the computation of the Krylov vectors inside the FGMRES method, than to the optimal behavior of the auxiliary problems. Secondarily, what we observe is indeed the optimal behavior with respect to the iteration discussed in Theorem~\ref{thm:clusteringasymptotic}. Indeed, the preconditioning routine becomes asymptotically better with the size of the problem, i.e., we get fewer iteration for bigger problems. Moreover, the decreasing of the $\alpha$ introduces just a  {latency} effect in the solution, i.e., the asymptotic regimes kicks in for slightly bigger problems when $\alpha$ is smaller,  {we stress that this is exactly the phenomenon described in Remark~\ref{remark:asymptotic_regime_of_generalized_eigenvalues} regarding the asymptotic relation between the value of $h$ going to zero, and the value of $\alpha$ being fixed independently {of} $h$.} { To overcome this limitation, one could decouple the system by neglecting the matrix $\alpha \bar{M}^{-1}$, i.e., the $(2,2)$ block in~\eqref{eq:the_linear_system}, thus obtaining the preconditioner
\begin{equation}\label{eq:decoupledpreconditioner}
	\mathcal{P}_{D} = \left[\begin{array}{cc|c}
	\bar{M} & O & \bar{K}^T \\
	&    & \\
	O & O & - \bar{M} \\
	&    & \\
	[-0.1em]      \hline        &    & \\
	\bar{K} & -\bar{M} & O
	\end{array}\right].
\end{equation}
By computation analogous to the one in Remark~\ref{rmk:alternative_preconditioners}, we find that the non-unit eigenvalues for this preconditioner are the ones of the matrix sequence $\left\lbrace I + \frac{\alpha}{h^4} M^{-1} K M^{-1} K^T \right\rbrace_N$. The non-unit eigenvalues tend to cluster at one whenever $\alpha h^{-4} \propto \alpha N^4 $ goes to zero. This means that $\mathcal{P}_D$ is efficient for small values of $\alpha$ and moderate values of $N$ and worsen for diverging values of $N$ (keeping fixed $\alpha$), indeed this is confirmed by the numerical test in Table~\ref{tab:decoupledpreconditioner}.}
\begin{table}[htbp]
\centering
{
\begin{tabular}{ccccccccc}
	\toprule
	& \multicolumn{8}{c}{GMRES preconditioned by $\mathcal{P}_D$} \\
	\cmidrule{3-9}
	\multirow{3}{*}{$\alpha =$ 	1.0e-09 }& N    & 147 & 675 & 2883 & 11907 & 48387 & 195075 & 783363 \\
	& IT   & 4 & 5 & 6 & $\dagger$ & $\dagger$ & $\dagger$ & $\dagger$  \\
	& T(s) & 1.0e-02 & 5.4e-03 & 1.7e-02 & - & - & - & -  \\
	\bottomrule
\end{tabular}}
\caption{{Poisson Control Problem. We report both the number of iterations, and the solution time for the $\mathcal{P}_D$ preconditioner in~\eqref{eq:decoupledpreconditioner}, compare these entries with the last block of rows of Table~\ref{tab:poissonproblem}}}\label{tab:decoupledpreconditioner}
\end{table}

\paragraph{Diffusion--Convection--Reaction} The second case we consider is the problem~\eqref{eq:thegeneralproblem} in which the costraint $e(y,u)$ is given by the Equation~\eqref{eq:adevection_diffusion_reaction}, with coefficients $r = 1$, and $\mathbf{c} = (2,3)$. The desired state is given by the sum of the two impulses
\begin{equation*}
y_d(x_1,x_2) = \frac{0.5}{0.07\sqrt{2 \pi}}e^{-\frac{(x_1-0.2)^2 + (x_2 - 0.2)^2}{2 (0.07)^2}} + \frac{0.8}{0.05\sqrt{2 \pi}}e^{-\frac{(x_1-0.6)^2 + (x_2 - 0.6)^2}{2 (0.05)^2}},
\end{equation*}
while the forcing term is given by
\begin{equation*}
z(x_1,x_2) = \sin(\pi x_1)\sin(\pi x_2).
\end{equation*}
We test the solution for regularization parameter $\alpha = \texttt{1.0e-03}, \texttt{1.0e-06}, \texttt{1.0e-09}$, and collect the results in Table~\ref{tab:advectiondiffusionreaction}.
\begin{table}[hbtp]
	\centering
	{\begin{tabular}{ll|llllll|llll}
			\toprule
			& & \multicolumn{6}{c}{} & \multicolumn{4}{c}{FGMRES}\\
			& & \multicolumn{6}{c}{GMRES} & \multicolumn{4}{c}{PCG/BiCGstab+IC/ILU} \\
			& & \multicolumn{2}{c}{$I_N$} &  \multicolumn{2}{c}{$\mathcal{P}_N$} &  \multicolumn{2}{c}{$\mathcal{P}_{\text{BCT}}$} &  \multicolumn{2}{c}{$\mathcal{P}_N$} &  \multicolumn{2}{c}{$\mathcal{P}_{\text{BCT}}$} \\
			$\alpha$ & N & IT & T(s) & IT & T(s) & IT & T(s) & IT & T(s) & IT & T(s) \\
			\midrule
			1.0e-03 & 147 & $\dagger$ & - & 5 & 8.7e-01 & 5 & \textbf{4.7e-03} & 5 & 7.7e-03 & 5 & 7.0e-03 \\
			& 675 & $\dagger$ & - & 5 & 9.6e-03 & 5 & 8.7e-03 & 5 & 7.1e-03 & 5 & \textbf{6.4e-03} \\
			& 2883 & $\dagger$ & - & 4 & 7.2e-02 & 4 & 2.7e-02 & 4 & 1.4e-01 & 4 & \textbf{9.5e-03} \\
			& 11907 & $\dagger$ & - & 3 & 8.6e-01 & 3 & 8.9e-02 & 3 & 4.8e-02 & 3 & \textbf{3.4e-02} \\
			& 48387 & $\dagger$ & - & 3 & 1.1e+00 & 3 & 4.4e-01 & 3 & 3.9e-01 & 3 & \textbf{2.5e-01} \\
			& 195075 & $\dagger$ & - & 2 & 1.7e+00 & 2 & 1.7e+00 & 2 & 1.7e+00 & 2 & \textbf{1.1e+00} \\
			& 783363 & $\dagger$ & - & 2 & 8.5e+00 & 2 & 8.9e+00 & 2 & \textbf{7.1e+00} & 2 & 7.4e+00 \\
			\cmidrule{3-12}
			1.0e-06 & 147 & $\dagger$ & - & 24 & 1.5e-02 & 24 & \textbf{1.4e-02} & 24 & 2.9e-02 & 24 & 2.4e-02 \\
			& 675 & $\dagger$ & - & 26 & 4.7e-02 & 26 & 4.6e-02 & 26 & 3.6e-02 & 27 & \textbf{3.3e-02} \\
			& 2883 & $\dagger$ & - & 24 & 1.7e-01 & 24 & 1.6e-01 & 24 & 7.2e-02 & 25 & \textbf{6.0e-02} \\
			& 11907 & $\dagger$ & - & 22 & 6.9e-01 & 22 & 7.0e-01 & 22 & 4.0e-01 & 24 & \textbf{2.9e-01} \\
			& 48387 & $\dagger$ & - & 19 & 2.8e+00 & 19 & 2.8e+00 & 19 & 2.5e+00 & 22 & \textbf{1.9e+00} \\
			& 195075 & $\dagger$ & - & 17 & 1.4e+01 & 17 & 1.4e+01 & 17 & 1.4e+01 & 18 & \textbf{1.2e+01} \\
			& 783363 & $\dagger$ & - & 14 & 5.9e+01 & 14 & 6.1e+01 & 14 & 7.9e+01 & 14 & \textbf{5.9e+01} \\
			\cmidrule{3-12}
			1.0e-09 & 147 & $\dagger$ & - & 38 & 3.8e-02 & 38 & \textbf{3.5e-02} & 38 & 4.2e-02 & 38 & 4.4e-02 \\
			& 675 & $\dagger$ & - & 73 & 1.5e-01 & 73 & 1.6e-01 & 73 & \textbf{1.2e-01} & 87 & 1.4e-01 \\
			& 2883 & $\dagger$ & - & 84 & 6.5e-01 & 73 & 6.5e-01 & 86 & \textbf{3.3e-01} & 73 & 3.7e-01 \\
			& 11907 & $\dagger$ & - & 94 & 3.5e+00 & 94 & 3.4e+00 & 97 & 2.1e+00 & 97 & \textbf{1.9e+00} \\
			& 48387 & $\dagger$ & - & 87 & 1.4e+01 & 87 & 1.4e+01 & 87 & 1.2e+01 & 87 & \textbf{1.1e+01} \\
			& 195075 & $\dagger$ & - & 77 & 6.8e+01 & 77 & 6.8e+01 & $\dagger$ & - & $\dagger$ & - \\
			& 783363 & $\dagger$ & - & 66 & 5.2e+02 & 66 & 5.4e+02 & $\dagger$ & - & $\dagger$ & - \\
			\bottomrule
	\end{tabular}}
	\caption{Diffusion--Convection--Reaction Control Problem. We compare both the number of iterations, and the solution time for the various preconditioners. Best timings are highlighted in bold face. When the method fails to converge, i.e., the method reaches the maximum number of iterations, a $\dagger$ is reported.  {The tolerances for the inner solvers are set to \texttt{1e-8}}}\label{tab:advectiondiffusionreaction}
\end{table}
The results are completely analogous to the one for the Poisson case. We observe a higher number of iteration that is due to the fact that we are using an  {asymptotic} argument both for the sequence $\mathcal{S}_N$, and for its block; see Proposition~\ref{pro:non-hermitian-case},  {and the discussion in Remark~\ref{remark:asymptotic_regime_of_generalized_eigenvalues} for the asymptotic relationship between $h$, and $\alpha$.}

\section{Conclusions and future developments}
\label{sec:conclusion_and_future_developments}

In this paper we have produced a characterization for the saddle--point matrices arising from the application of the discretize--then--optimize approach to quadratic optimization problems with elliptic PDE constraints highlighting the presence of an hidden Generalized Locally Toeplitz structure, i.e., we have proposed an analysis that is sharper and more informative than the one that can be obtained by looking only at the saddle--point structure. We have produced a localization of the spectrum in three intervals, up to a number of outliers infinitesimal in the dimension of the problem, and used this characterization to produce an asymptotically optimal preconditioner, i.e., a preconditioner that is independent of the value of the regularization parameter $\alpha$, and whose performance increases for finer grids.

We plan to extend this analysis  { in order that it can cover more} general constraints, i.e.,  {we would like to }discuss also the case of  {sparse} optimization, and  {bounded} controls. Moreover, the GLT spectral analysis techniques we are using have been recently extended for becoming tools for the  {fast} and  {reliable} computation of generalized eigenvalues see, e.g.,~\cite{MR3816229,MR3894727}, since we have analyzed the structure of the eigenvectors of our preconditioned problems (Proposition~\ref{prop:from_preconditioned_to_eigenvalue_problem}), we plan to investigate the possible application of  {deflation} techniques to further accelerate our iterative methods.

\bigskip
{\bf Acknowledgment.} We are thankful to Prof. S. Serra--Capizzano for the insightful discussions on the spectral distribution results, and to the referee whose suggestion have been extremely helpful in improving the presentation of the material.

\bibliography{saddlespectral}

\end{document}